\documentclass[12pt, leqno]{amsart}

\usepackage{amssymb}

\usepackage[all]{xy}

\theoremstyle{plain}

\newtheorem{theorem}{Theorem}[section]
\newtheorem{lemma}[theorem]{Lemma}
\newtheorem{proposition}[theorem]{Proposition}
\newtheorem{corollary}[theorem]{Corollary}

\theoremstyle{definition}

\newtheorem{remark}[theorem]{Remark}

\newtheorem{example}[theorem]{Example}

\newcommand\bA{{\mathbb A}}

\newcommand\bG{{\mathbb G}}

\newcommand\bZ{{\mathbb Z}}

\newcommand\cC{{\mathcal C}}

\newcommand\cF{{\mathcal F}}

\newcommand\cL{{\mathcal L}}
\newcommand\cM{{\mathcal M}}

\newcommand\cO{{\mathcal O}}

\newcommand\rD{{\rm D}}
\newcommand\rF{{\rm F}}
\newcommand\rH{{\rm H}}

\newcommand\rR{{\rm R}}
\newcommand\rV{{\rm V}}
\newcommand\rX{{\rm X}}

\newcommand\tcC{\widetilde{\mathcal C}}
\newcommand\tcF{\widetilde{\mathcal F}}
\newcommand\tcL{\widetilde{\mathcal L}}

\newcommand\wA{\widehat{A}}
\newcommand\wB{\widehat{B}}

\newcommand\wvarphi{\widehat{\varphi}}

\renewcommand\deg{{\rm deg}}

\newcommand\id{{\rm id}}
\newcommand\ind{{\rm ind}}

\newcommand\res{{\rm res}}
\newcommand\rk{{\rm rk}}

\newcommand\Coh{{\rm Coh}}
\newcommand\HVec{{\rm HVec}}
\newcommand\tmod{\text -{\rm mod}}
\newcommand\tMod{\text -{\rm Mod}}
\newcommand\QCoh{{\rm QCoh}}
\newcommand\UVec{{\rm UVec}}
\renewcommand\Vec{{\rm Vec}}

\newcommand\Aut{{\rm Aut}}
\newcommand\Coker{{\rm Coker}}

\newcommand\End{{\rm End}}
\newcommand\Ext{{\rm Ext}}

\newcommand\Gal{{\rm Gal}}
\newcommand\GL{{\rm GL}}

\newcommand\Hom{{\rm Hom}}

\newcommand\Ind{{\rm Ind}}

\newcommand\Ker{{\rm Ker}}

\newcommand\Pic{{\rm Pic}}

\newcommand\Pro{{\rm Pro}}

\newcommand\Sym{{\rm Sym}}

\numberwithin{equation}{section}

\title[Homogeneous vector bundles]
{Homogeneous vector bundles over abelian varieties 
via representation theory}

\author{Michel Brion}

\date{}

\begin{document}

\begin{abstract}
Let $A$ be an abelian variety over a field. 
The homogeneous (or translation-invariant) vector bundles 
over $A$ form an abelian category $\HVec_A$; 
the Fourier-Mukai transform yields an equivalence of 
$\HVec_A$ with the category of coherent sheaves with finite support 
on the dual abelian variety. In this paper, we develop an alternative 
approach to homogeneous vector bundles, based on the equivalence 
of $\HVec_A$ with the category of finite-dimensional representations 
of a commutative affine group scheme (the ``affine fundamental group'' 
of $A$). This displays remarkable analogies between homogeneous 
vector bundles over abelian varieties and representations of split
reductive algebraic groups.
\end{abstract}

\maketitle

\tableofcontents

\section{Introduction}
\label{sec:int}

Consider an abelian variety $A$ over a field $k$. 
A vector bundle over $A$ is said to be homogeneous, 
if it is invariant under pullback by translations in $A$;
for instance, the homogeneous line bundles are 
parameterized by $\Pic^0(A) = \wA(k)$, 
where $\wA$ denotes the dual abelian variety. 
The homogeneous vector bundles have been extensively 
studied when $k$ is algebraically closed. As shown by Mukai, 
they form an abelian category $\HVec_A$ which is equivalent 
to the category of coherent sheaves with finite support on 
$\wA$ via the Fourier-Mukai transform. Moreover, any 
homogeneous vector bundle has a canonical decomposition
\begin{equation}\label{eqn:idec}
E = \bigoplus_{L \in \wA(k)} L \otimes U_L,
\end{equation}
where each $U_L$ is a unipotent vector bundle, i.e., 
an iterated extension of trivial line bundles; the unipotent 
bundles form an abelian subcategory $\UVec_A$, 
equivalent to that of coherent sheaves on $\wA$ 
with support at the origin (see \cite[\S 4]{Mukai} 
for these results). Also, the homogeneous vector bundles 
can be characterized as the iterated extensions 
of algebraically trivial line bundles (a result of Miyanishi,
see \cite[\S 2]{Miyanishi}), or as the numerically flat 
vector bundles (this is due to Langer, see \cite[\S 6]{La12}).

In this paper, we develop an alternative approach to 
homogeneous vector bundles via representation theory,
over an arbitrary ground field $k$. The starting point  
is again a result of Miyanishi in \cite{Miyanishi}: 
every such bundle is an associated vector bundle 
$G \times^H V \to G/H = A$, 
for some extension of algebraic groups 
$1 \to H \to G \to A \to 1$ with $H$ affine and some 
finite-dimensional $H$-module $V$. 
To handle all homogeneous vector bundles simultaneously, 
we take the inverse limit of all such extensions; this yields 
the universal extension of $A$ by an affine $k$-group scheme,
\begin{equation}\label{eqn:iuniv} 
1 \longrightarrow H_A \longrightarrow G_A
\longrightarrow A \longrightarrow 1,
\end{equation}
where $G_A$ is a quasi-compact $k$-group scheme
(see \cite[IV.8.2.3]{EGA}).
We show that the associated vector bundle construction 
$V \mapsto G_A \times^{H_A} V$ yields an equivalence
of categories
\begin{equation}\label{eqn:iequiv} 
\cL_A :  H_A\tmod \stackrel{\simeq}{\longrightarrow} \HVec_A, 
\end{equation}
where $H_A\tmod$ denotes the category of finite-dimensional 
$H_A$-modules (Theorem \ref{thm:equiv}).
Moreover, $\cL_A$ induces isomorphisms of extension groups
\begin{equation}\label{eqn:iext}
\Ext^i_{H_A}(V,W) \stackrel{\simeq}{\longrightarrow}
\Ext^i_A(\cL_A(V),\cL_A(W)) 
\end{equation}
for all $i \geq 0$ and all finite-dimensional $H_A$-modules
$V,W$ (Theorem \ref{thm:ext}). In particular, homogeneity 
is preserved under extensions of vector bundles.

The $k$-group schemes $G_A$ and $H_A$ are not of finite type, 
but their structure is rather well-understood (see 
\cite[\S 3.3]{Br18} and \S \ref{subsec:univ}). 
In particular, $G_A$ is commutative and geometrically integral; 
its formation commutes with base change under arbitrary 
field extensions. Moreover, $H_A$ is an extension of a unipotent 
group scheme $U_A$ by the group scheme of multiplicative 
type $M_A$ with Cartier dual $\wA(\bar{k})$. Also, 
$U_A$ is a vector group of dimension $g = \dim(A)$ 
if $k$ has characteristic $0$; in positive characteristics, 
$U_A$ is profinite (see \S \ref{subsec:univ} again). By Theorem 
\ref{thm:unip}, the equivalence of categories (\ref{eqn:iequiv}) 
induces an equivalence 
\[ U_A\tmod \stackrel{\simeq}{\longrightarrow} \UVec_A. \]

Homogeneous vector bundles are also preserved under tensor 
product and duality; clearly, these operations on $\HVec_A$
correspond via (\ref{eqn:iequiv}) to the tensor product and 
duality of $H_A$-modules (the description of these operations
in terms of Fourier-Mukai transform is more involved, 
see \cite[Thm.~4.12]{Mukai}). Also, homogeneity is preserved 
by pushforward and pullback under isogenies; 
we show that these correspond again to natural operations on 
modules, in view of a remarkable invariance property of 
the extension (\ref{eqn:iuniv}). More specifically, 
let $\varphi: A \to B$ be an isogeny of abelian varieties, 
with kernel $N$. Then $\varphi$ induces an isomorphism 
$G_A \stackrel{\simeq}{\longrightarrow} G_B$ 
and an exact sequence
\begin{equation}\label{eqn:iH_AH_B}
0 \longrightarrow H_A \longrightarrow H_B
\longrightarrow N \longrightarrow 0
\end{equation}
(Proposition \ref{prop:fun}). Moreover, the pushforward 
(resp.~pullback) under $\varphi$ yields an exact functor
$\varphi_* : \HVec_A \to \HVec_B$
(resp.~$\varphi^* : \HVec_B \to \HVec_A$), which may be 
identified with the induction, $\ind_{H_A}^{H_B}$ 
(resp.~the restriction, $\res_{H_A}^{H_B}$) via the
equivalences of categories $\cL_A$, $\cL_B$ 
(Theorem \ref{thm:isogeny}). 

In particular, the homogeneous vector bundle
$\varphi_*(\cO_A)$ on $B$ corresponds to 
the $H_B$-module $\cO(N)$, where $H_B$ acts via
the regular representation of its quotient group $N$; thus,
the indecomposable summands of $\varphi_*(\cO_A)$
correspond to the blocks of the finite group scheme $N$. 
When applied to the $n$th relative Frobenius morphism
in characteristic $p$, this yields a refinement of a recent 
result of Sannai and Tanaka (\cite[Thm.~1.2]{ST}, see 
Remark \ref{rem:ST} for details). 

More generally, the category $H_A\tmod$ decomposes
into blocks associated with irreducible representations
of $H_A$, since every extension of two non-isomorphic 
such representations is trivial (see \S \ref{subsec:fun}). 
This translates into a decomposition of the category 
$\HVec_A$ which generalizes (\ref{eqn:idec}); the line bundles 
$L \in \wA(k)$ are replaced with irreducible homogeneous 
vector bundles, parameterized by the orbits of the absolute 
Galois group in $\wA(\bar{k})$ (see \S \ref{subsec:fun}
again). We obtain a description of the block 
$\HVec_{A,x}$ associated with a point $x \in \wA(\bar{k})$, 
which takes a simpler form when the residue field $K := k(x)$ 
is separable over $k$: then 
$\HVec_{A,x} \simeq \UVec_{A_K} = \HVec_{A_K,0}$
(Theorem \ref{thm:unip}).  

In addition to these results, the paper contains developments
in two independent directions. Firstly, we consider the universal 
affine covers of abelian varieties in the setting of quasi-compact 
group schemes (\S \ref{subsec:qc}). These covers are interesting 
objects in their own right; over an algebraically closed field, 
they were introduced and studied by Serre as projective covers 
in the abelian category of commutative proalgebraic groups 
(see \cite[\S 9.2]{Se60}, and \cite[\S 3.3]{Br18} 
for a generalization to an arbitrary ground field). 
They also occur in very recent work of Ayoub on motives 
(see \cite[\S 5.9]{Ayoub}); moreover, the corresponding affine 
fundamental group, i.e., the group $H_A$ in the exact sequence 
(\ref{eqn:iuniv}), coincides with the $S$-fundamental group 
of $A$ as defined in \cite{La11} (see \cite[\S 6]{La12}). 
We show that the group $G_A$ is the projective cover
of $A$ in the abelian category of commutative quasi-compact
group schemes; as a consequence, that category has enough 
projectives (Proposition \ref{prop:proj}). This builds on results 
of Perrin about (not necessarily commutative) quasi-compact 
group schemes, see \cite{Pe75, Pe76}, and follows 
a suggestion of Grothendieck in \cite[IV.8.13.6]{EGA}.
\footnote{``Les seuls pro-groupes alg\'ebriques rencontr\'es en pratique
jusqu'\`a pr\'esent \'etant en fait essentiellement affines, il y aura 
sans aucun doute avantage \`a substituer \`a l'\'etude des groupes
pro-alg\'ebriques g\'en\'eraux (introduits et \'etudi\'es par Serre 
\cite{Se60}) celle des sch\'emas en groupes quasi-compacts sur
$k$, dont la d\'efinition est conceptuellement plus simple.''}

The second development, in Section \ref{sec:rep}
(which can be read independently of the rest of the paper), 
investigates the representation theory of a commutative 
affine group scheme $H$ over a field $k$. This is easy and 
well-known when $k$ is algebraically closed (see e.g. 
\cite[I.3.11]{Jantzen}); the case of a perfect field $k$ follows 
readily by using Galois descent and the splitting 
$H = U \times M$, where $U$ is unipotent and 
$M$ of multiplicative type. But the affine fundamental 
group $H_A$ over an imperfect field admits no such splitting 
(Lemma \ref{lem:H_A}); this motivates our study of a topic 
which seems to have been unexplored.

This representation-theoretic approach displays
remarkable analogies between homogeneous vector
bundles over an abelian variety $A$ and over a full flag
variety $X = G/B$, where $G$ is a split reductive algebraic
group and $B$ a Borel subgroup: $G_A,H_A$ play similar roles 
as $G,B$; the group $\wA(\bar{k})$ is replaced with the
weight lattice, and the Galois group with the Weyl group... 
But these analogies are incomplete, as the combinatorics 
associated with root data of reductive groups have
no clear counterpart on the side of abelian varieties;
also, the block decompositions behave very differently.

\medskip

\noindent
{\bf Notation and conventions}. Throughout this paper, we fix a ground 
field $k$ of characteristic $p \geq 0$. We choose an algebraic closure 
$\bar{k}$ of $k$, and denote by $k_s$ the separable closure of $k$ 
in $\bar{k}$. The Galois group of $k_s/k$ is denoted by $\Gamma$.

We consider schemes over $k$ unless otherwise mentioned; 
morphisms and products of schemes are understood to be
over $k$ as well. Given a scheme $S$ and a field extension $K/k$, 
we denote by $S_K$ the scheme obtained from $S$ by the 
corresponding base change. We freely identify line bundles
(resp.~vector bundles) over a scheme with invertible sheaves
(resp.~locally free sheaves).

A \emph{variety} is a separated, geometrically integral scheme
of finite type. An \emph{algebraic group} is a group scheme 
of finite type. 

We denote by $A$ a nonzero abelian variety, and set 
$g := \dim(A)$; as usual, the group law of $A$ is denoted 
additively, with neutral element $0 \in A(k)$. We use the
notation $n_A$ for the multiplication by an integer $n$ in $A$,
and $A[n]$ for its (scheme-theoretic) kernel. The dual abelian
variety is denoted by $\wA$.

\section{Universal affine covers and homogeneous bundles}
\label{sec:univ}

\subsection{Commutative quasi-compact group schemes}
\label{subsec:qc}

Let $\cC$ be the category with objects the commutative algebraic 
groups, and morphisms the homomorphisms of group schemes; 
then $\cC$ is an abelian category 
(see \cite[VIA.5.4.2]{SGA3}). Denote by $\cL$ the full 
subcategory of $\cC$ with objects the affine (or equivalently, linear) 
algebraic groups; then $\cL$ is an abelian subcategory of $\cC$, 
stable under taking subobjects, quotients and extensions. We say 
that $\cL$ is a \emph{Serre subcategory} of $\cC$. The full 
subcategory $\cF$ of $\cC$ with objects the finite group schemes 
is a Serre subcategory of $\cL$.

Next, consider the category $\tcC$ with objects the commutative
quasi-compact group schemes; the morphisms in $\tcC$ are
still the homomorphisms of group schemes. By \cite[V.3.6]{Pe75},
$\tcC$ is an abelian category. The monomorphisms in $\tcC$
are exactly the closed immersions; the epimorphisms are exactly 
the faithfully flat morphisms (see \cite[V.3.2, V.3.4]{Pe75}).
Also, note that the quotient morphism $G \to G/H$ is a fpqc
torsor for any $G \in \tcC$ and any subgroup scheme $H \subset G$
(see \cite[IV.5.1.7.1]{SGA3}). Using fpqc descent, it follows  
that $\cC$ is a Serre subcategory of $\tcC$.

Every commutative affine group scheme is an object of $\tcC$; 
these objects form a full subcategory $\tcL$, which is again 
a Serre subcategory in view of \cite[VIB.9.2, VIB.11.17]{SGA3}.
Likewise, the profinite group schemes form a Serre subcategory
$\tcF$ of $\tcL$.

By the main theorem of \cite{Pe75}, every quasi-compact group 
scheme $G$ (possibly non-commutative) is the filtered inverse 
limit of algebraic quotient groups $G_i$ with transition 
functions affine for large $i$; equivalently, $G$ is
\emph{essentially affine} in the sense of \cite[IV.8.13.4]{EGA}. 
We now record a slightly stronger version of this result in 
our commutative setting:

\begin{lemma}\label{lem:lift}

Let $G \in \tcC$.

\begin{enumerate}

\item[{\rm (i)}] For any epimorphism $f : G \to H$ in $\tcC$,
where $H$ is affine (resp.~profinite), there exists 
an affine (resp.~profinite) subgroup scheme $H' \subset G$ 
such that the composition $H' \to G \to H$ is an epimorphism.

\item[{\rm (ii)}] There is an exact sequence in $\tcC$
\[ 0 \longrightarrow H \longrightarrow G 
\longrightarrow A \longrightarrow 0, \]
where $H$ is affine and $A$ is an abelian variety.

\item[{\rm (iii)}] $G$ is the filtered inverse limit of
its quotients $G/H$, where $H$ runs over the affine  
subgroup schemes of $G$ such that $G/H$ is algebraic.

\end{enumerate}

\end{lemma}

\begin{proof}
(i) We may choose an affine subgroup scheme $L \subset G$
such that $G/L$ is algebraic. Then $f$ induces an epimorphism
$G/L \to H/f(L)$, and $H/f(L)$ is affine. By the lifting
property for the pair $(\cC,\cL)$ (see \cite[Lem.~3.1]{Br18}), 
there exists an affine subgroup scheme $K \subset G/L$
such that the composition $K \to G/L \to H/f(L)$ is an
epimorphism. Consider the preimage $H'$ of $K$ in $G$.
Then $H'$ is affine (as an extension of $K$ by $L$) 
and the composition $H' \to G \to H$ is an epimorphism.
This shows the assertion for affine quotients. 
The assertion for profinite quotients follows similarly from
the lifting property for the pair $(\cC,\cF)$, obtained in 
\cite[Thm.~1.1]{Br15} (see also \cite[Thm.~3.2]{Lucchini}).

(ii) Consider the neutral component $G^0$ of $G$; this is
a connected subgroup scheme and the quotient $G/G^0$ is
pro-\'etale (see \cite[V.4.1]{Pe75}). Moreover, by 
\cite[V.4.3.1]{Pe75}, there is an exact sequence
in $\tcC$
\[ 0 \longrightarrow H' \longrightarrow G^0 
\longrightarrow A \longrightarrow 0, \]
where $H'$ is affine and $A$ is an abelian variety.
In view of (i), there is a profinite subgroup scheme
$H'' \subset G$ such that the composition $H'' \to G \to G/G^0$
is an epimorphism. Then $H:= H' + H''$ is an affine subgroup
scheme of $G$, and $G/H$ is a quotient of $A$, hence an
abelian variety.

(iii) Let $H$ be as in (ii). Then the subgroup schemes
$H' \subset H$ such that $H/H'$ is algebraic form a filtered
inverse system, and their intersection is trivial. Moreover,
for any such subgroup scheme $H'$, the quotient $G/H'$ is algebraic
(as an extension of $A$ by $H/H'$). This yields the assertion
in view of \cite[II.3.1.1]{Pe75}.  
\end{proof}

In particular, every commutative quasi-compact group scheme 
$G$ is a filtered inverse limit of algebraic quotient groups with 
affine transition morphisms.

We now consider the \emph{pro category} $\Pro(\cC)$: its objects
(the \emph{pro-algebraic groups}) are the filtered inverse systems 
of objects of $\cC$, and the morphisms are defined by 
\[ \Hom_{\Pro(\cC)}
( \lim_{\leftarrow} G_i, \lim_{\leftarrow} H_j) 
:= \lim_{\leftarrow, j} \lim_{\rightarrow, i} 
\Hom_{\cC}(G_i,H_j). \]
Recall that $\Pro(\cC)$ is an abelian category having
enough projectives; moreover, the natural functor
\[ F : \cC \longrightarrow \Pro(\cC) \] 
yields an equivalence of $\cC$ with the Serre subcategory 
$\cC'$ of $\Pro(\cC)$ consisting of artinian objects (see e.g. 
\cite[I.4]{Oort} for these facts). 
Also, recall that for any object $X$ of $\Pro(\cC)$, 
the artinian quotients $X_i$ of $X$ form a filtered inverse 
system, and the resulting map $X \to \lim_{\leftarrow} X_i$ 
is an isomorphism (see e.g. \cite[V.2.2]{DG}). 
As a consequence, the restriction of $F$ to $\cL$ extends
to an equivalence of categories 
$\tcL \stackrel{\simeq}{\longrightarrow} \Pro(\cL)$; 
also, note that $\Pro(\cL)$ is a Serre subcategory of $\Pro(\cC)$.

Denote by $\tcC'$ the full subcategory of $\Pro(\cC)$ consisting
of those objects $X$ such that there exists an exact sequence
in $\Pro(\cC)$
\begin{equation}\label{eqn:exte} 
0 \longrightarrow Y \longrightarrow X \longrightarrow A
\longrightarrow 0, 
\end{equation}
where $Y \in \Pro(\cL)$ and $A$ is (the image under $F$ of)
an abelian variety.

\begin{lemma}\label{lem:equiv}

\begin{enumerate}

\item[{\rm (i)}] The objects of $\tcC'$ are exactly the
pro-algebraic groups that are isomorphic to essentially
affine objects.

\item[{\rm (ii)}]  
$\tcC'$ is a Serre subcatgory of $\Pro(\cC)$.

\item[{\rm (iii)}] 
Sending each $G \in \tcC$ to the filtered inverse system of 
its algebraic quotients extends to an exact functor 
$S: \tcC \to \tcC'$, which is an equivalence of categories. 

\end{enumerate}

\end{lemma}

\begin{proof}
(i) Let $X \in \Pro(\cC)$ be an essentially affine object.
Then there is an exact sequence in $\Pro(\cC)$
\[ 0 \longrightarrow Y \longrightarrow X 
\longrightarrow Z \longrightarrow 0, \]
where $Y$ is affine and $Z$ is algebraic. Thus, $Z$
is an extension of an abelian variety by an affine group
scheme. Hence so is $X$; it follows that $X$ is an object 
of $\tcC'$.

Conversely, every object of $\Pro(\cC)$ that lies in
an extension (\ref{eqn:exte}) is isomorphic to the
essentially affine object consisting of the inverse system
of quotients $X/Y'$, where $Y'$ is a subobject of $Y$.

(ii) Let $X \in \tcC'$ and consider a subobject $X'$
of $X$ in $\Pro(\cC)$. Then we have a commutative diagram
with exact rows
\[ \xymatrix{
0 \ar[r] & Y' \ar[r] \ar[d] & X' \ar[r] \ar[d] & A'  
\ar[r] \ar[d] & 0 \\
0 \ar[r] & Y \ar[r]  & X  \ar[r] & A \ar[r] & 0, \\
} \]
where the vertical arrows are monomorphisms. It follows that
$Y' \in \Pro(\cL)$, and $A'$ is an extension of an abelian
variety by a finite group scheme (Lemma \ref{lem:lift}).
As a consequence, $X' \in \tcC'$. Moreover, $X/X'$ is an extension
of $A/A'$ by $Y/Y'$, and hence is an object of $\tcC'$ as well.
So $\tcC'$ is stable by subobjects and quotients.
To show the stability by extensions, it suffices to check
that $\tcC'$ contains all objects $X$ of $\Pro(\cC)$ which lie 
in an exact sequence
\[ 0 \longrightarrow A \longrightarrow X
\longrightarrow Y \longrightarrow 0, \]
where $A$ is an abelian variety and $Y \in \Pro(\cL)$. But this
follows from the fact that the pair $(\Pro(\cC),\Pro(\cL))$ 
satisfies the lifting property (see \cite[Cor.~2.12]{Br18}).

(iii) Note that $S(G)$ is essentially affine for any
$G \in \tcC$. Next, we define $S(f)$ for any morphism
$f : G \to H$ in $\tcC$. If $H \in \cC$, then $f$ factors
through a morphism $f' : G' \to H$ for any sufficiently 
large algebraic quotient $G'$ of $G$; we then define
$S(f)$ as the image of $f'$ in 
\[ \lim_{\rightarrow} \Hom_{\cC}(G',H) = 
\Hom_{\Pro(\cC)}(S(G),S(H)). \] 
This extends to an arbitrary $H \in \tcC$ by using
the equality
\[ \Hom_{\Pro(\cC)}(G, \lim_{\leftarrow} H') = 
\lim_{\leftarrow} \Hom_{\Pro(\cC)}(G,H'). \]

To show that $S$ satisfies our assertions, recall
from \cite[IV.8.2.3]{EGA} that for any essentially affine
object $X = (X_i)$ of $\Pro(\cC)$, the inverse limit of the 
$X_i$ exists in $\tcC$; denote this quasi-compact group scheme 
by $L(X)$. Moreover, $L$ extends to a equivalence from the
full subcategory of $\Pro(\cC)$ consisting of the essentially
affine objects, to the category $\tcC$ (see 
\cite[IV.8.13.5, IV.8.13.6]{EGA}). By construction,
$L \circ S$ is isomorphic to the identity functor of
$\tcC$. This yields the desired statement, except
for the exactness of $S$. By \cite[IV.8.13.6]{EGA} again, 
$L$ commutes with products; equivalently, $L$ is additive.
It follows that $S$ is additive as well. To show that it
is exact, it suffices to check that $S$ preserves kernels
and cokernels. In turn, it suffices to show that $S$
preserves monomorphisms and epimorphisms. But this follows
readily from the stability of $\tcC'$ under subobjects
and quotients.
\end{proof}

\begin{proposition}\label{prop:proj}

\begin{enumerate}

\item[{\rm (i)}] For any object $G$ of $\tcC'$, 
the projective cover of $G$ in $\Pro(\cC)$ is an object 
of $\tcC'$ as well.

\item[{\rm (ii)}] Every indecomposable projective of
$\Pro(\cC)$ is an object of $\tcC'$.

\item[{\rm (iii)}] $\tcC$ has enough projectives.

\end{enumerate}

\end{proposition}

\begin{proof}
(i) View $G$ as an extension of an abelian variety
$A$ by an affine group scheme $H$. This readily yields 
an isomorphism of projective covers 
$P(G) \simeq P(H) \oplus P(A)$ with an obvious
notation. Moreover, $P(H)$ is affine, and $P(A)$
is an extension of $A$ by an affine group scheme
(see \cite[Prop.~3.3]{Br18}). Thus, $P(G) \in \tcC'$.

(ii) This follows from the fact that every indecomposable 
projective object of $\Pro(\cC)$ is the projective cover 
of an algebraic group (see e.g. \cite[V.2.4.3]{DG}).

(iii) By (i), the abelian category $\tcC'$ has enough 
projectives. Hence so does $\tcC$ in view of Lemma 
\ref{lem:equiv}.
\end{proof}

\subsection{Universal affine covers of abelian varieties}
\label{subsec:univ}

Let $A$ be an abelian variety. By Proposition \ref{prop:proj}, 
$A$ has a projective cover $G_A$ in $\tcC$.
The resulting exact sequence,
\begin{equation}\label{eqn:G_A} 
0 \longrightarrow H_A \longrightarrow G_A
\stackrel{f_A}{\longrightarrow} A \longrightarrow 0, 
\end{equation} 
is the universal extension of $A$ by an affine group scheme, 
as follows from \cite[\S 3.3]{Br18} combined with 
Lemma \ref{lem:equiv}. More specifically, for any commutative
affine group scheme $H$, there is an isomorphism
\begin{equation}\label{eqn:homext} 
\Hom_{\tcL}(H_A,H) \stackrel{\simeq}{\longrightarrow}
\Ext^1_{\tcC}(A,H) 
\end{equation}
given by pushout of the extension (\ref{eqn:G_A}). 
Also, by Lemma \ref{lem:equiv} again and 
\cite[\S \S 2.3, 3.4]{Br18}, the projective objects of
$\tcC$ are exactly the products $P \times G_B$, where 
$P$ is a projective object of $\tcL$ and $B$ is an 
abelian variety.

Note that $f_A : G_A\to A$ is the filtered inverse limit of 
all \emph{anti-affine extensions} of $A$, i.e., of all algebraic groups 
$G$ equipped with an epimorphism $f : G \to A$ such that $\Ker(f)$
is affine and $\cO(G) = k$ (see \cite[Lem.~2.14]{Br18}). 
In particular, $\cO(G_A) = k$ as well. Using \cite[4.2.2]{Pe75},
it follows that $G_A$ is geometrically integral.

Next, we show that (\ref{eqn:G_A}) is the universal
extension of $A$ by a (possibly non-commutative)
affine group scheme:

\begin{theorem}\label{thm:univ}
Let $G$ be a quasi-compact group scheme, and $f: G \to A$
a faithfully flat morphism of group schemes with affine kernel 
$H$. Then there exist unique morphisms of group schemes 
$\varphi: G_A \to G$, $\psi : H_A \to H$ such that 
we have a commutative diagram with exact rows
\[ \xymatrix{
0 \ar[r] & H_A \ar[r] \ar[d]^{\psi} & G_A 
\ar[r]^-{f_A} \ar[d]^{\varphi} & A  \ar[r] \ar[d]^{\id} & 0 \\
1 \ar[r] & H \ar[r]  & G  \ar[r]^-{f} & A \ar[r] & 1. \\
} \]
Moreover, $\varphi$ factors through the center of 
the neutral component $G^0$. 
\end{theorem}

\begin{proof}
As $G$ is quasi-compact, we may
freely use the results of \cite[V.3]{Pe75} on the
representability of the fpqc quotients of $G$ and 
its subgroup schemes. Also, recall the affinization
theorem (see \cite[4.2.2]{Pe75}): $G$ has 
a largest normal subgroup scheme $N$ such that 
$G/N$ is affine; moreover, $N$ is geometrically
integral and contained in the center of $G^0$. 
In particular, $N$ is an object of $\tcC$.
Let $B$ denote the scheme-theoretic image of
$N$ under $f: G \to A$. Then $B$ is an abelian 
subvariety of $A$; moreover, $A/B$ is isomorphic
to a quotient of $G/N$, and hence is affine. 
So $B = A$, and hence $f$ restricts to an
epimorphism $g : N \to A$ in $\tcC$ with 
affine kernel. This readily yields the existence
of $\varphi$, $\psi$.

For the uniqueness, just note that $\varphi$
factors through $N$, since every affine quotient
of $G_A$ is trivial.
\end{proof}

\begin{proposition}\label{prop:field}
The formation of $G_A$ commutes with base change under
arbitrary field extensions.
\end{proposition}

\begin{proof}
Let $k'/k$ be a field extension. For any anti-affine extension
$f: G \to A$, the base change $f_{k'} : G_{k'} \to A_{k'}$
is an anti-affine extension again, as
$\cO(G_{k'}) = \cO(G) \otimes_k k'$ (see e.g. 
\cite[VIB.11.1]{SGA3}). Also, since $G_A = \lim_{\leftarrow} G$
(where $G$ runs over the above anti-affine extensions of $A$)
and base change commutes with filtered inverse limits, we obtain 
$(G_A)_{k'} = \lim_{\leftarrow} G_{k'}$. This yields a commutative 
diagram with exact rows
\[ \xymatrix{
0 \ar[r] & (H_A)_{k'} \ar[r] \ar[d]^{\psi} & (G_A)_{k'} \ar[r] \ar[d]^{\varphi} 
& A_{k'}  \ar[r] \ar[d]^{\id} & 0 \\
0 \ar[r] & H_{A_{k'}} \ar[r]  & G_{A_{k'}}  \ar[r] & A_{k'} \ar[r] & 0. \\
} \]
As a consequence, $\Coker(\varphi) \simeq \Coker(\psi)$
is affine. Since $(G_A)_{k'}$ is anti-affine, it follows that $\varphi$
is an epimorphism. As $G_{A_{k'}}$ is projective in $\tcC_{k'}$,
this yields an isomorphism 
$(G_A)_{k'} \simeq  G_{A_{k'}} \times \Ker(\varphi)$.
In particular, $\Ker(\varphi)$ is a quotient of $(G_A)_{k'}$.
But $\Ker(\varphi) \simeq \Ker(\psi)$ is affine, and hence trivial.
\end{proof}

The commutative affine group scheme $H_A$ lies in a unique
exact sequence 
\begin{equation}\label{eqn:H_A}
0 \longrightarrow M_A \longrightarrow H_A
\longrightarrow U_A \longrightarrow 0,
\end{equation}
where $M_A$ is of multiplicative type and $U_A$ is
unipotent; when $k$ is perfect, this exact sequence
has a unique splitting (see \cite[IV.3.1.1]{DG}). We now
describe the structure of $M_A$:

\begin{lemma}\label{lem:H_A}

\begin{enumerate}

\item[{\rm (i)}] For any field extension $k'/k$, there is
a natural isomorphism
$\Hom_{\tcL}(H_{A_{k'}},\bG_{m,k'}) \simeq \wA(k')$.

\item[{\rm (ii)}] The Cartier dual of $M_A$ is the 
$\Gamma$-module $\wA(\bar{k})$.

\item[{\rm (iii)}] The group scheme $H_A$ is not
algebraic. When $k$ is imperfect, the exact
sequence (\ref{eqn:H_A}) does not split.

\end{enumerate}

\end{lemma}

\begin{proof}
(i) By Proposition \ref{prop:field}, we may assume
that $k' = k$. Then (\ref{eqn:homext}) yields 
a natural isomorphism 
$\Hom_{\tcL}(H_A,\bG_m) \simeq \Ext^1_{\tcC}(A,\bG_m)$,
which implies the asssertion in view of the Weil-Barsotti 
formula (see e.g. \cite[III.17, III.18]{Oort}).

(ii) This follows readily from (i) by taking $k' = \bar{k}$.

(iii) Assume that $H_A$ is an algebraic group. Then
$\Hom_{\tcL}(H_{A_{\bar{k}}},\bG_{m,\bar{k}})$
is a finitely generated abelian group. But $\wA(\bar{k})$ 
is not finitely generated, since it has nonzero $\ell$-torsion
for any prime $\ell \neq p$; a contradiction.

Next, assume that the extension (\ref{eqn:H_A}) splits. 
Then the resulting isomorphism 
$H_A \simeq M_A \times U_A$
and Proposition \ref{prop:field} yield an isomorphism 
\[ \Hom_{\tcL}(H_{A_{k'}},\bG_{m,k'})
\simeq \Hom_{\tcL}((M_A)_{k'},\bG_{m,k'}) \]
for any field extension $k'/k$. As $M_A$
is of multiplicative type, it follows that the
natural map
$\Hom_{\tcL}(H_{A_{k_s}},\bG_{m,k_s})
\to \Hom_{\tcL}(H_{A_{\bar{k}}},\bG_{m,\bar{k}})$
is an isomorphism. In view of (i), this yields the equality
$\wA(k_s) = \wA(\bar{k})$. Choose a non-empty
open affine subscheme $U$ of $\wA$; then 
$U(k_s) = U(\bar{k})$ as well. By Noether normalization, 
there exists a finite surjective morphism
$F: U \to \bA_k^g$. If $k$ is imperfect, choose 
$t \in \bar{k} \setminus k_s$; then the fiber of $F$ at 
$(t,0,\ldots,0)\in \bA^g_k(\bar{k})$ contains no 
$k_s$-rational point, a contradiction.
\end{proof}

We now turn to the structure of $U_A$. If $p = 0$, 
then $U_A$ is the vector group 
associated with the dual vector space of 
$\rH^1(A,\cO_A)$ (see \cite[Lem.~3.8]{Br18}); 
thus, $U_A \simeq (\bG_a)^g$. If $p > 0$, then $U_A$ 
is profinite; more specifically, $U_A$ is the largest unipotent 
quotient of the profinite fundamental group of $A$,  
$\lim_{\leftarrow} A[n]$ (as follows from 
\cite[Thm.~3.10]{Br18}). Thus, $U_A$ is the largest
unipotent quotient of its pro-$p$ part,
$A[p^{\infty}] := \lim_{\leftarrow} A[p^n]$. 

Next, recall from \cite[\S 15]{Mumford} that
\begin{equation}\label{eqn:ptors}
A[p^n]_{\bar{k}} \simeq 
(\bZ/p^n\bZ)^r_{\bar{k}} \times (\mu_{p^n})^r_{\bar{k}}
\times G_n, 
\end{equation}
where $r = r_A$ is an integer independent of $n$ (called 
the $p$-rank of $A$), and $G_n$ is a unipotent infinitesimal
$\bar{k}$-group scheme of order $p^{2n(g - r)}$. 
The abelian variety $A$ is said to be \emph{ordinary}
if it has maximal $p$-rank, i.e., $r = g$ (see \cite[\S 2.1]{ST}
for further characterizations of ordinary abelian varieties). 
It follows that
\[ (U_A)_{\bar{k}} \simeq \lim_{\leftarrow} \,
(\bZ/p^n\bZ)^r_{\bar{k}} \times G_nn \simeq
(\bZ_p)^r_{\bar{k}} \times \lim_{\leftarrow} \, G_n. \]
As a consequence, $A$ is ordinary if and only if
$U_A$ is pro-\'etale; then 
\begin{equation}\label{eqn:U_A}
U_A(k_s) \simeq (\bZ_p)^g.
\end{equation}

\begin{remark}\label{rem:formal}
When $k$ is algebraically closed, the groups 
$\Ext^1_{\cC}(A,H)$, where $H$ is a commutative affine 
algebraic group, have been determined by Wu in \cite{Wu}. 
His results may be recovered from the above description 
of $H_A$ in view of the isomorphism (\ref{eqn:homext}).

Also, this description can be interpreted in terms of
formal groups associated with $\wA$, via Cartier duality 
which yields an anti-equivalence of $\tcL$ with 
the category of commutative formal $k$-groups (see 
\cite[VIIB.2]{SGA3}, that we will freely use as a general 
reference for formal groups). Under that equivalence, 
the exact sequence (\ref{eqn:H_A}) corresponds to 
an exact sequence of formal groups,
\[ 0 \longrightarrow \rD(U_A) \longrightarrow \rD(H_A)
\longrightarrow \rD(M_A) \longrightarrow 0, \]
where $\rD(U_A)$ is infinitesimal and $\rD(M_A)$
is \'etale; this exact sequence splits if and only if
$k$ is perfect. By Lemma \ref{lem:H_A} (ii), $\rD(M_A)$ 
corresponds to the $\Gamma$-group $\wA(\bar{k})$ 
under the equivalence of \'etale formal groups with 
$\Gamma$-groups. 

Also, if $p = 0$, then $\rD(U_A)$ is the infinitesimal 
formal group associated with the commutative 
Lie algebra $H^1(A,\cO_A)$. Since the latter is the
Lie algebra of $\wA$, we may identify $\rD(U_A)$ 
with the infinitesimal formal neighborhood of the
origin in $\wA$. 

This still holds if $p > 0$: then the above
infinitesimal formal group is isomorphic to
$\lim_{\rightarrow} \Ker(\rF^n_{\wA})$, where 
$\rF^n_{\wA} : \wA \to \wA^{(p^n)}$
denotes the $n$th relative Frobenius morphism.
Moreover, we have canonical isomorphisms
\[ \rD(\lim_{\rightarrow} \Ker(\rF^n_{\wA})) \simeq
\lim_{\leftarrow} \rD(\Ker(\rF^n_{\wA})) \simeq
\lim_{\leftarrow} \Ker(\rV^n_A), \]
where $\rV^n_A : A^{(p^n)} \to A$
denotes the $n$th Verschiebung, and the second 
isomorphism follows from duality between Frobenius and 
Verschiebung. Also, we have a canonical exact sequence
\[ 0 \longrightarrow \Ker(\rF^n_A) \longrightarrow
A[p^n] \longrightarrow \Ker(\rV^n_A)
\longrightarrow 0 \]
for any $n \geq 1$; as a consequence, 
$\lim_{\leftarrow} \Ker(\rV^n_A)$ may be identified
with the largest unipotent quotient of $A[p^{\infty}]$, 
i.e., with $U_A$. 
\end{remark}

\subsection{Applications to homogeneous vector bundles}
\label{subsec:hvb}

Let $\pi : E \to A$ be a vector bundle. For any scheme $S$,
consider the set $G(S)$ consisting of the pairs 
$(\varphi,a)$ satisfying the following conditions:

\begin{enumerate}

\item[{\rm (i)}] $\varphi :E_S \to E_S$ is an isomorphism
of $S$-schemes, and $a \in A(S)$.

\item[{\rm (ii)}] The diagram
\[ \xymatrix{
E_S \ar[r]^{\varphi}\ar[d]_{\pi_S} & E_S \ar[d]^{\pi_S} \\
A_S  \ar[r]^{\tau_a} & A_S \\
} \]
commutes, where $\tau_a$ denotes the translation by
$a$ in $A_S$.

\item[{\rm (iii)}] The isomorphism of $A_S$-schemes
$E_S \to \tau_a^*(E_S)$ induced by $\varphi$, is 
an isomorphism of vector bundles.

\end{enumerate}

Alternatively, we may view $G(S)$ either as the set of pairs 
$(a,\psi)$, where $a \in A(S)$ and $\psi : E_S \to \tau_a^*(E_S)$ 
is an isomorphism of vector bundles (see \cite[p.~72]{Miyanishi}),
or as the group of automorphisms of the $S$-scheme $E_S$ 
which lift translations in $A_S$ and commute with the action
of $\bG_{m,S}$ by multiplication on fibers.

Clearly, $G(S)$ is a group for componentwise multiplication
of pairs $(\varphi,a)$;
moreover, the assignment $S \mapsto G(S)$ extends to 
a group functor that we still denote by $G$, or $G_E$ 
to emphasize its dependence in $E$. The second
projection yields a morphism of group functors 
$f : G \to A$; the kernel of $f$ is the group functor
$H = H_E$ of automorphisms of the vector bundle $E$.

\begin{lemma}\label{lem:cG}

\begin{enumerate}

\item[{\rm (i)}] $H$ is a smooth connected
affine algebraic group.

\item[{\rm (ii)}] $G$ is an algebraic group.

\item[{\rm (iii)}] $E$ is equipped with a $G$-linearization.

\end{enumerate}

\end{lemma}

\begin{proof}
(i) Note that $H$ is the group of invertibles 
of the monoid functor of endomorphisms of the vector
bundle $E$, and this monoid functor is represented by 
an affine space. So the assertion follows from
\cite[II.2.3.6]{DG}.

(ii) The group functor of $\bG_m$-equivariant automorphisms 
of $E$ is represented by a group scheme $\Aut^{\bG_m}_E$, 
locally of finite type; moreover, we have an exact sequence
of group schemes
\[ 1 \longrightarrow H \longrightarrow \Aut^{\bG_m}_E
\longrightarrow \Aut_A \]
(see \cite[Prop.~6.3.2]{BSU}).
Viewing $A$ as a subgroup scheme of $\Aut_A$ via
its action by translation, we may identify $G$ with
the pullback of $A$ in $\Aut^{\bG_m}_E$; this yields
the assertion.

(iii) This follows e.g. from \cite[I.6.5.3]{SGA3}.
\end{proof}

By Lemma \ref{lem:cG}, there is an exact sequence 
of algebraic groups
\begin{equation}\label{eqn:cG}
1 \longrightarrow H_E \longrightarrow G_E
\stackrel{f}{\longrightarrow} A.
\end{equation}
We say that $E$ is \emph{homogeneous} if $f$ 
is faithfully flat; equivalently, (\ref{eqn:cG}) is 
right exact. Since $f$ is a morphism of algebraic
groups and $A$ is smooth, this amounts to the 
following condition:
for any $a \in A(\bar{k})$, the translation by $a$
in $A_{\bar{k}}$ lifts to an automorphism of the
$\bar{k}$-variety $E_{\bar{k}}$, linear on fibers.
Equivalently, $E_{\bar{k}} \simeq \tau_a^*(E_{\bar{k}})$
for any $a \in A(\bar{k})$. 

As a consequence, if $E$ is homogeneous, then so is 
the dual vector bundle $E^{\vee}$; if in addition $F$ 
is a homogeneous vector bundle over $A$, then 
$E \oplus F$ and $E \otimes F$ are homogeneous as well. 
Thus, the homogeneous vector bundles form 
a full subcategory of the category $\Vec_A$ of vector
bundles over $A$, stable under finite direct sums, duals 
and tensor products. We denote this additive tensor
subcategory by $\HVec_A$.

Let $E$ be a homogeneous vector bundle over $A$.
Then the algebraic group $G_E$ is smooth and
connected, since so are $H_E$ and $A$. 
This group acts transitively on $A$ via $f$,
and the stabilizer of the origin $0$ equals $H_E$. 
The $G_E$-linearization of $E$ restricts to an action 
of $H_E$ on the fiber $E_0$ via a representation
\[ \rho : H_E \longrightarrow \GL(E_0). \]
Moreover, the morphism $G_E \times E_0 \to E$
given by the action, factors through an isomorphism
\[ G_E \times^{H_E} E_0 
\stackrel{\simeq}{\longrightarrow} E, \]
where the left-hand side denotes the quotient
of $G_E \times E_0$ by the action of $H_E$ 
via $h \cdot (g,x) := (gh^{-1}, \rho(h) x)$;
this is the vector bundle over $A$ associated with
the $H_E$-torsor $G_E \to A$ and the 
$H_E$-representation in $E_0$.

Also, by Theorem \ref{thm:univ}, 
we have a commutative diagram with exact rows
\[ \xymatrix{
1 \ar[r] & H_A \ar[r] \ar[d]^{\psi} & G_A 
\ar[r]^-{f_A} \ar[d]^{\varphi} & A  \ar[r] \ar[d]^{\id} & 1 \\
1 \ar[r] & H_E \ar[r]  & G_E  \ar[r]^-{f} & A \ar[r] & 1 \\
} \]
for unique morphisms $\varphi = \varphi_E$, $\psi = \psi_E$; 
moreover, $\varphi_E$ factors through the center of $G_E$. 
This yields a representation 
$\rho \circ \psi : H_A \to \GL(E_0)$
and an isomorphim $G_A\times^{H_A} E_0 \simeq E$. 
In particular, $E$ is $G_A$-linearized.

Conversely, given a finite-dimensional representation
$\rho : H_A \to \GL(V)$, the quotient 
$H := H_A/\Ker(\rho)$ is algebraic and lies in
an exact sequence 
\[ 0 \longrightarrow  H \longrightarrow  G 
\longrightarrow A \longrightarrow 0,\]
where $G := G_A/\Ker(\rho)$ is algebraic as well.
Thus, we may form the associated vector bundle
$\cL_{G/H}(V)$ (see \cite[I.5.8, I.5.15]{Jantzen}): 
this is a $G$-linearized vector bundle on $A$, and hence 
is homogeneous. We denote this vector bundle by 
\[ \cL_A(V) = G_A \times^{H_A} V \longrightarrow A . \]
Its fiber at $0$ is $V$, on which $H_A$ acts via the
representation $\rho$. 

Denote by $H_A$-mod the category of finite-dimensional
representations of $H_A$; this is an abelian tensor category.
We may now state:

\begin{theorem}\label{thm:equiv}
The above assignments $E \mapsto E_0$, 
$V \mapsto G_A \times^{H_A} V$ extend to exact functors 
\[ \cM_A : \HVec_A \longrightarrow H_A\tmod, \quad
\cL_A : H_A\tmod \longrightarrow \HVec_A \]
which are quasi-inverse equivalences of additive tensor 
categories.
\end{theorem}

\begin{proof}
Let $E$, $F$ be homogeneous vector bundles
over $A$. Since they both are $G_A$-linearized, 
the finite-dimensional vector space $\Hom_{\Vec_A}(E,F)$ 
(consisting of the morphisms of vector bundles 
$\gamma: E \to F$) is equipped with a linear representation 
of $G_A$ (see \cite[I.6.6.2]{SGA3}). As
$\cO(G_A) = k$, this representation is trivial, i.e.,
every $\gamma$ as above is $G_A$-equivariant. 
As a consequence, the restriction $\gamma_0 : E_0 \to F_0$
is $H_A$-equivariant, where $H_A$ acts 
on $E_0$ and $F_0$ via the above representations.
The assignment $\gamma \mapsto \gamma_0$
defines the functor $\cM_A$ on morphisms.

Given $E$, $F$ as above, consider the image $G$ 
of the product morphism 
\[ (\varphi_E,\varphi_F) : 
G_A \longrightarrow G_E \times G_F. \] 
Then $G$ is an algebraic quotient of $G_A$, and lies 
in an exact sequence of algebraic groups
\[ 0 \longrightarrow H \longrightarrow G 
\longrightarrow A \longrightarrow 0, \]
for some algebraic quotient $H$ of $H_A$. 
Moreover, $E \simeq G \times^H E_0$, $F \simeq G \times^H F_0$,
and $G$ acts trivially on $\Hom_{\Vec_A}(E,F)$. This yields
an isomorphism
\[ \Hom_{\Vec_A}(E,F) \simeq 
\Hom_{\Vec_A}^G(G \times^H E_0,G \times^H F_0). \]
The right-hand side is contained in the set 
$\Hom^G(G \times^H E_0,G \times^H F_0)$
of $G$-equivariant morphisms of schemes, 
which is identified with
$\Hom^H(E_0,G \times^H F_0)$ via restriction; this identifies
the subset $\Hom_{\Vec_A}^G(G \times^H E_0,G \times^H F_0)$ 
with the subset of $H$-equivariant linear maps $E_0 \to F_0$. 
As a consequence, the functor $\cM_A$ is fully faithful.
It is essentially surjective, since $\cM_A(\cL_A(V)) = V$
for any finite-dimensional $H_A$-module $V$.

Next, note that every morphism of $H_A$-modules
$u: V \to W$ yields a morphism of associated vector bundles 
$\cL_A(u) : \cL_A(V) \to \cL_A(W)$ such that
$\cL_A(u)_0 = u$, as may be checked by factoring
both representations $H_A \to \GL(V)$, $H_A \to \GL(W)$
through a common algebraic quotient of $H_A$. This
defines the functor $\cL_A$ on morphisms, and shows that
it is quasi-inverse to $\cM_A$. Clearly, $\cL_A$ and $\cM_A$
are exact and preserve finite direct sums, duals and 
tensor products.
\end{proof}

\begin{corollary}\label{cor:abe}
The subcategory $\HVec_A$ of $\Vec_A$ is abelian
and stable under direct summands.
\end{corollary}

Next, we obtain a key vanishing result for the coherent
cohomology of $G_A$:

\begin{proposition}\label{prop:van}
$\rH^i(G_A, \cO_{G_A}) = 0$ for any $i \geq 1$.
\end{proposition}

\begin{proof}
To simplify the notation, we set 
$G := G_A$, $H := H_A$ and $f := f_A$. 
Since the morphism $f: G \to A$ is affine, we have 
$R^i f_*(\cO_G) = 0$ for all $i \geq 1$ in view of
\cite[III.1.3.2]{EGA}. This yields isomorphisms
\[ \rH^i(G, \cO_G) \cong \rH^i(A,f_*(\cO_G)) 
\quad (i \geq 0). \]
Next, recall that $G \simeq \lim_{\leftarrow} G'$,
where the limit is taken over the filtered inverse system of 
anti-affine extensions $f' : G' \to A$. Thus, we have an isomorphism 
of quasi-coherent sheaves of $\cO_A$-algebras
\[ f_*(\cO_G) \simeq \lim_{\rightarrow} f'_*(\cO_{G'}). \]
Since cohomology commutes with direct limits, this yields
in turn isomorphisms
\[ \rH^i(G, \cO_G) \simeq 
\lim_{\rightarrow} \rH^i(A,f'_*(\cO_{G'})) 
\simeq \lim_{\rightarrow} \rH^i(G',\cO_{G'}) 
\quad (i \geq 0). \] 
As each $G'$ is anti-affine, the coherent cohomology ring
$\rH^*(G', \cO_{G'})$ is the exterior algebra over
$\rH^1(G', \cO_{G'})$ (see \cite[Thm.~1.1]{Br13}).
As a consequence, $\rH^*(G,\cO_G)$ is the exterior
algebra over $\rH^1(G,\cO_G)$. Thus, it suffices to
show the vanishing of $\rH^1(G,\cO_G)$.  For this, we may
assume that $k$ is algebraically closed, in view of
Proposition \ref{prop:field}.

By adapting the argument of \cite[Lem.~9.2]{Totaro}
(see also \cite[Thm.~VII.5]{Se97}), one may check that
the canonical map
\[ \Ext^1_{\cC}(G',\bG_a) \longrightarrow 
\rH^1(G',\cO_{G'})^{G'} \]
is an isomorphism, where the right-hand side
denotes the subspace of $G'$-invariants in
$\rH^1(G',\cO_{G'})$. As $\rH^1(G',\cO_{G'})$ 
is a (rational) $G'$-module (see e.g. \cite[Lem.~2.1]{Br13}) 
and $G'$ is anti-affine, this yields compatible isomorphisms
$\Ext^1_{\cC}(G',\bG_a) \simeq \rH^1(G',\cO_{G'})$,
and hence an isomorphism
\[ \Ext^1_{\Pro(\cC)}(G,\bG_a) \simeq \rH^1(G,\cO_G), \]
by using \cite[V.2.3.9]{DG}. The desired vanishing
follows from this, since $G$ is projective in $\Pro(\cC)$.

Alternatively, the vanishing of $\rH^1(G,\cO_G)$
can be obtained as follows. Let again $f' : G' \to A'$
be an anti-affine extension and consider the Albanese
morphism $\alpha : G' \to A'$; then $f'$ is the composition
of $\alpha$ with an isogeny $A' \to A$. If $p = 0$ and
$\alpha$ lifts to an epimorphism $G' \to E(A')$
(the universal vector extension of $A'$), then
$H^1(G', \cO_{G'}) = 0$ by 
\cite[Prop.~4.1, Prop.~4.3]{Br13}; this yields the desired
vanishing. On the other hand, if $p > 0$, then the pullback 
$\alpha^* :  \rH^*(A',\cO_{A'}) \to \rH^*(G', \cO_{G'})$ 
is an isomorphism by \cite[Cor.~4.2]{Br13}.    
The commutative diagram of multiplication maps
\[ \xymatrix{
G \ar[r] \ar[d]^{p_G} & G' \ar[r]^-{\alpha} \ar[d]^{p_{G'}} 
& A' \ar[d]^{p_{A'}} \\
G \ar[r]  & G'  \ar[r]^-{\alpha} & A' \\
} \]
yields a commutative diagram of pullbacks
\[ \xymatrix{
\rH^1(A',\cO_{A'}) \ar[r]^{\simeq} \ar[d]^{p_{A'}^*} & 
\rH^1(G',\cO_{G'}) \ar[r] \ar[d]^{p_{G'}^*} & 
\rH^1(G,\cO_G) \ar[d]^{p_G^*} \\
\rH^1(A',\cO_{A'}) \ar[r]^{\simeq}  
& \rH^1(G',\cO_{G'}) \ar[r] 
& \rH^1(G,\cO_G). \\
} \]
Moreover, $p_G$ is an isomorphism by \cite[Lem.~3.4]{Br18};
hence so is $p_G^*$. Thus, it suffices to show that 
$p_{A'}^* = 0$. But this follows from the isomorphism
$\Ext^1_{\cC}(A',\bG_a) \cong \rH^1(A',\cO_{A'})$ 
(see \cite[Thm.~VII.5]{Se97})
together with the equality $p_{A'}^* = (p_{\bG_a})_* = 0$ 
in $\Ext^1_{\cC}(A',\bG_a)$.
\end{proof}

\begin{theorem}\label{thm:ext}
Let $V$, $W$ be finite-dimensional $H_A$-modules. 
Then the map 
\[ \Ext^i_{H_A}(V,W) \longrightarrow 
\Ext^i_A(\cL_A(V),\cL_A(W)) \]
is an isomorphism for any $i \geq 0$,
where the right-hand side denotes the higher extension 
group of coherent sheaves on $A$.
\end{theorem}

\begin{proof}
Denoting by $V^{\vee}$ the dual $H_A$-module of $V$, 
the statement can be reformulated as follows: the map
\[ \rH^i(H_A,V^{\vee} \otimes_k W) \longrightarrow 
\rH^i(A, \cL_A(V^{\vee} \otimes_k W)) \]
is an isomorphism for any $i \geq 0$ (use
\cite[I.4.2, I.4.4]{Jantzen} and the compatibility 
of $\cL_A$ with duals and tensor products). This holds for 
$i = 0$ in view of Theorem \ref{thm:equiv}. Thus, 
it suffices to extend $\cL_A$ to an exact functor
from $H_A$-modules to quasi-coherent sheaves on $A$, 
which takes injective modules to acyclic sheaves; then
the statement will follow from a degenerate case of
Grothendieck's spectral sequence. 

We first check that $\cL_A$ extends uniquely to an exact 
functor, commuting with filtered direct limits, 
from the category $H_A$-Mod of all $H_A$-modules 
(not necessarily finite-dimensional), to the category
$\QCoh_A$ of quasi-coherent sheaves on $A$. Indeed, 
$H_A$-Mod is a Grothendieck category and its noetherian
objects are exactly those of $H_A$-mod; as a consequence,
$H_A$-Mod is equivalent to the ind category $\Ind(H_A\tmod)$. 
Likewise, $\QCoh_A$ is equivalent to $\Ind(\Coh_A)$. 
So the desired assertion follows from \cite[6.1.9, 8.6.8]{KS}. 

We still denote the extended functor by
\[ \cL_A : H_A\tMod \longrightarrow \QCoh_A. \]
We now check that $\cL_A$ takes injectives to acyclics. 
By \cite[I.3.10]{Jantzen}, every injective $H_A$-module 
is a direct summand of a direct sum of copies of $\cO(H_A)$; 
hence it suffices to show that 
$\rH^i(A,\cL_A(\cO(H_A))) = 0$ for all $i \geq 1$.
In view of (the proof of) Proposition \ref{prop:van},
it suffices in turn to check that
\begin{equation}\label{eqn:direct}
\cL_A(\cO(H_A)) \simeq f_*(\cO_{G_A}),
\end{equation}
where $H_A$ acts on $\cO(H_A)$ via the regular representation.

Recall that $\cO(H_A) = \lim_{\leftarrow} \cO(H)$ and
$f_*(\cO_{G_A}) = \lim_{\rightarrow} f'_*(\cO_G)$, 
where both limits run over all anti-affine extensions 
$0 \longrightarrow H \longrightarrow G
\stackrel{f}{\longrightarrow} A \longrightarrow 0$. Also, recall
that $\cL_A(V)$ is defined as the associated sheaf 
$\cL_{G/H}(V)$ for any finite-dimensional $H$-module
$V$, and hence for any $H$-module. 
By \cite[I.5.18]{Jantzen}, there are compatible isomorphisms 
\[ \cL_A(\cO(H)) \stackrel{\simeq}{\longrightarrow}
f'_*(\cO_G). \]
This yields the desired isomorphism (\ref{eqn:direct}).
\end{proof}

\begin{corollary}\label{cor:ext}
The subcategory $\HVec_A $ of $\Vec_A$ is stable 
under extensions.
\end{corollary}

\section{Representations of commutative affine group schemes}
\label{sec:rep}

Throughout this section, we consider linear representations 
of a fixed commutative affine $k$-group scheme $H$. 
We use the book \cite{Jantzen} as a general reference for
representation theory. 

\subsection{Irreducible representations}
\label{subsec:irr}

The aim of this subsection is to construct the irreducible 
representations of $H$ (Proposition \ref{prop:irr}), and classify 
them up to equivalence (Proposition \ref{prop:gal}).

Define a \emph{character} of $H$ as a morphism of
$\bar{k}$-group schemes
\[ \chi : H_{\bar{k}} \longrightarrow \bG_{m,\bar{k}}. \]
The characters form an abelian group that we denote by 
$\rX(H)$. Every such character $\chi$ may be viewed as 
an element of 
$\cO(H_{\bar{k}}) = \cO(H) \otimes_k \, \bar{k}$. 
Thus, $\chi$ is defined over a smallest finite subextension 
$K/k$ of $\bar{k}/k$: the \emph{field of definition}, 
$K = k(\chi)$ (generated by the coordinates of $\chi$ in 
a $\bar{k}$-basis of $\cO(H_{\bar{k}})$ consisting of elements 
of $\cO(H)$). So $\chi$ defines a morphism of $K$-group 
schemes $H_K \to \bG_{m,K},$ or equivalently a morphism 
of $k$-group schemes
\begin{equation}\label{eqn:irr} 
\rho(\chi) : H \longrightarrow \rR_{K/k}(\bG_{m,K}), 
\end{equation}
where $\rR_{K/k}$ denotes the Weil restriction of scalars:
for any scheme $S$, the group of $S$-points
$\rR_{K/k}(\bG_{m,K})(S)$ is the unit group of
the $K$-algebra $\cO(S) \otimes_k K$
(see \cite[A.5]{CGP} for generalities on Weil restriction). 
Thus, we may view $\rho(\chi)$ as a linear representation 
of $H$ in the $k$-vector space $K$, in which 
$\rR_{K/k}(\bG_{m,K})$ acts by multiplication.

\begin{proposition}\label{prop:irr}

\begin{enumerate}

\item[{\rm (i)}] The representation $\rho(\chi)$ 
is irreducible, with commutant algebra $k(\chi)$.

\item[{\rm (ii)}] Every irreducible representation 
of $H$ is isomorphic to $\rho(\chi)$ for some
$\chi \in \rX(H)$.

\end{enumerate}

\end{proposition}

\begin{proof}
We argue as in the proof of \cite[9.4]{Waterhouse}.
Let $\rho: H \to \GL(V)$ be a finite-dimensional 
representation. Consider the dual representation 
$\rho^{\vee}$ of $H$ in $V^{\vee}$, and the corresponding 
comodule map 
\[ \Delta_{\rho^{\vee}} : 
V^{\vee} \longrightarrow V^{\vee} \otimes_k \cO(H). \]
Since $H$ is commutative, $\Delta_{\rho^{\vee}}$ 
is equivariant for the $H$-representation on $V^{\vee}$, 
and that on $V^{\vee} \otimes_k \cO(H)$ via its action on 
$V^{\vee}$. Also, the matrix coefficients of $\rho^{\vee}$ 
span a subspace $C = C(\rho^{\vee}) \subset \cO(H)$ 
which is stable under the comultiplication 
$\Delta : \cO(H) \to \cO(H) \otimes_k \cO(H)$; hence
$C$ is a finite-dimensional sub-coalgebra. Note that 
$C$ is co-commutative and has a co-unit, since these
properties hold for $\cO(H)$. Thus, the dual vector
space $C^{\vee}$ is a finite-dimensional commutative
algebra with unit, and $V$ is a $C^{\vee}$-module. 
Moreover, $C^{\vee}$ acts on $V$ by $H$-invariant 
endomorphisms, since the transpose map 
$V^{\vee} \to V^{\vee} \otimes_k C$ is equivariant for 
the $H$-representations as above. Also, by 
\cite[II.2.2.4]{DG}, a subspace $W \subset V^{\vee}$ 
is $H$-stable if and only if 
$\Delta_{\rho^{\vee}}(W) \subset W \otimes_k C$. 

We now assume that $\rho$ is irreducible; then so is 
$\rho^{\vee}$, and hence $V^{\vee}$ is a simple $C$-comodule. 
Equivalently, $V$ is a simple $C^{\vee}$-module. 
This yields an isomorphism of $C^{\vee}$-modules 
$V \simeq K$ for some quotient field $K$ of $C^{\vee}$; 
moreover, $H$ acts linearly on $K$ and commutes with the
$K$-action by multiplication. Thus, $\rho$ is identified
with a morphism of $k$-group schemes
$H \to \rR_{K/k}(\bG_{m,K})$, which corresponds
to a morphism of $K$-group schemes 
$H_K \to \bG_{m,K}$, or equivalently to a character 
$\chi \in \rX(H)$, defined over $K$.
If $\chi$ is defined over a smaller subfield $L$ 
containing $k$, then the corresponding morphism 
of $L$-group schemes $H_L \to \bG_{m,L}$
yields a factorization of $\rho$ through 
a morphism $H \to \rR_{L/k}(\bG_{m,L})$.
In particular, $H$ stabilizes the subspace 
$L$ of $K$. As $\rho$ is irreducible, it follows
that $L = K$, i.e., $K = k(\chi)$ and
$\rho = \rho(\chi)$. This shows (ii).

To prove (i), we may assume that $\rho(\chi)$
is a monomorphism, and hence view $H$ as
a $k$-subgroup scheme of $\rR_{K/k}(\bG_{m,K})$,
where $K = k(\chi)$. For the representation
of $\rR_{K/k}(\bG_{m,K})$ in $K$, the co-algebra
of matrix coefficients is easily checked to
be $K^{\vee}$ (the $k$-linear dual of $K$).
As a consequence, the co-algebra $C$ of matrix
coefficients of $\rho(\chi)$ is a quotient
of $K^{\vee}$; equivalently, $C^{\vee}$
is a $k$-subalgebra of $K$, and hence a subfield, 
say $L$. Moreover, we have a commutative
diagram of morphisms of $k$-algebras
\[ \xymatrix{
\Sym(K^{\vee}) \ar[r] \ar[d] 
& \cO(\rR_{K/k}(\bG_{m,K})) \ar[d] \\
\Sym(L^{\vee}) \ar[r]  & \cO(H), \\
} \]
which translates into a commutative diagram of
morphisms of $k$-schemes
\[ \xymatrix{
H \ar[r] \ar[d] &  \rR_{L/k}(\bA^1_L) \ar[d] \\
\rR_{K/k}(\bG_{m,K}) \ar[r]  & \rR_{K/k}(\bA^1_K), \\
} \]
where the vertical arrows and the bottom
horizontal arrow are immersions. Hence so is 
the top horizontal arrow. As a consequence,
$H$ acts linearly on $K$ via a morphism
to $R_{L/k}(\bG_{m,L})$, and $\chi$
is defined over $L$. Thus, $L = K = C^{\vee}$. 
Using again the characterization of the 
$H$-stable subspaces of $V^{\vee}$ in terms 
of the comodule map, it follows that $V$ 
is irreducible. In particular, its commutant
algebra is a division algebra $D$ of finite 
dimension over $k$, say $d$. We have 
$d \geq [K:k]$ as $D \supset K$, and
$d \leq [K:k]$ as $K$ is a $D$-module.
So equality holds, and $D = K$.
\end{proof}

Next, consider the absolute Galois group 
\[ \Gamma := \Gal(k_s/k) = \Aut(\bar{k}/k) ,\]
and its continuous action on the character group 
$\rX(H)$; denote this action by 
$(\gamma,\chi) \mapsto \gamma \cdot \chi$. 
Note that $k(\gamma \cdot \chi) = \gamma k(\chi)$ 
for any $\gamma \in \Gamma$ and $\chi \in \rX(H)$. 

Also, recall that $H$ lies in a unique exact sequence of 
commutative affine group schemes
\begin{equation}\label{eqn:H}
0 \longrightarrow M \longrightarrow H \longrightarrow U
\longrightarrow 0, 
\end{equation}
where $M$ is of multiplicative type and $U$ is unipotent;
the formation of $M$, $U$ commutes with base change 
under field extensions. For any such extension $k'/k$ 
with $k'$ perfect, the induced exact sequence 
\[ 0 \longrightarrow M_{k'} \longrightarrow H_{k'} 
\longrightarrow U_{k'} \longrightarrow 0 \]
has a unique splitting; in particular, (\ref{eqn:H}) splits
over $\bar{k}$. Since every character of $U$
is trivial, this yields a $\Gamma$-equivariant isomorphism 
\[ \rX(H) \longrightarrow \rX(M), 
\quad \chi \longmapsto \chi \vert_M,\]
where the $\Gamma$-group $\rX(M)$ may be identified
with the Cartier dual $\rD(M)$.
Finally, recall from \cite[II.3.3.7, II.3.4.3]{DG} that 
every representation of $M$ is completely reducible.

\begin{lemma}\label{lem:def}

Let $\chi \in \rX(H)$ and denote by $\Gamma_{\chi}$
its stabilizer in $\Gamma$. Set $K := k(\chi)$, 
$\eta := \chi \vert_M$ and $L := k(\eta)$.

\begin{enumerate}

\item[{\rm (i)}] $L = K \cap k_s$.

\item[{\rm (ii)}] $L$ is the fixed
subfield of $\Gamma_{\chi}$ in $k_s$.

\item[{\rm (iii)}] 
$ \res_M^H \, \rho(\chi) \simeq [K : L] \, \rho(\eta)$.

\end{enumerate}

\end{lemma}

\begin{proof}
(i) If $p = 0$ then the assertion follows readily from the
isomorphism $H \simeq M \times U$. Thus, we may assume 
$p > 0$. 

Clearly, $\eta$ is defined over $K$; that is,
$L \subset K$. Also, as $M$ is of multiplicative 
type, its characters are all defined over $k_s$; in particular,
$L$ is separable over $k$. To complete the proof,
it suffices to show that $K^{p^n} \subset L$
for $n \gg 0$. For this, we may replace $k$ with $L$,
and hence assume that $\eta \in \Hom_{\tcL}(M,\bG_m)$.
Since $\Hom_{\tcL}(U,\bG_m) = 0$, we have an exact sequence
\[ 0 \longrightarrow \Hom_{\tcL}(H,\bG_m)
\longrightarrow \Hom_{\tcL}(M,\bG_m)
\longrightarrow \Ext^1_{\tcL}(U,\bG_m). \]
Also, $U$ is the filtered inverse limit of its algebraic
quotients $U'$; in view of \cite[V.2.3.9]{DG}, this yields
\[ \Ext^1_{\tcL}(U,\bG_m) \simeq 
\lim_{\rightarrow} \Ext^1_{\tcL}(U',\bG_m). \]
Since each $U'$ is a unipotent algebraic group,
$\Ext^1_{\tcL}(U',\bG_m)$ is killed by a power of $p$.
Thus, so is every element of $\Ext^1_{\tcL}(U,\bG_m)$;
hence $p^n \eta$ extends to a morphism
$H \to \bG_m$ for $n \gg 0$. This implies the assertion. 

(ii) We view $\eta$ as an element of 
$\cO(M_{k_s}) = \cO(M) \otimes_k \, k_s$. Write
accordingly 
$\eta = \sum_{i=1}^n f_i \otimes z_i$,
where $f_1,\ldots,f_n \in \cO(M)$ are linearly
independent over $k$ and $z_1,\ldots,z_n \in k_s$.
Then $k(\eta) = k(z_1,\ldots,z_n)$.
Also, the stabilizer in $\Gamma$ of $\chi$
equals that of $\eta$, and hence consists
of the $\gamma$ such that $\gamma z_i = z_i$
for $i = 1,\ldots,n$. Thus, 
$\gamma \in \Gamma_{\chi}$ if and only if
$\gamma$ fixes $L$ pointwise.
So the assertion follows from Galois theory.

(iii) This is a direct consequence of the definition 
of $\rho(\chi)$.
\end{proof}

\begin{proposition}\label{prop:gal}
Let $\chi,\chi' \in \rX(H)$. Then we have
$\rho(\chi') \simeq \rho(\chi)$ if and only if
$\chi' \in \Gamma \cdot \chi$.
\end{proposition}

\begin{proof}
Assume that $\chi' = \gamma \cdot \chi$ for some 
$\gamma \in \Gamma$. We may choose
a finite subextension $K'/k$ which is stable under
$\Gamma$ and contains $k(\chi)$. Then
$\Gamma$ acts by automorphisms on the
$k$-group scheme $\rR_{K'/k}(\bG_{m,K'})$,
and $\gamma$ restricts to an isomorphism 
\[ \rR_{k(\chi)/k}(\bG_{m,k(\chi)})
\stackrel{\simeq}{\longrightarrow} 
\rR_{k(\chi')/k}(\bG_{m,k(\chi')}) \] 
which intertwines $\rho(\chi)$ and $\rho(\chi')$. 

Conversely, assume that  
$\rho(\chi')$ is isomorphic to $\rho(\chi)$. 
Then the commutant algebras are isomorphic
as well, hence $k(\chi') \simeq k(\chi)$
as extensions of $k$ (Proposition \ref{prop:irr}). 
In view of Lemma \ref{lem:def}, it follows that 
$\rho(\chi' \vert_M) \simeq \rho(\chi \vert_M)$. 
Consider the diagonalizable $k_s$-group scheme
$M_{k_s}$; then its representation 
$\rho(\chi \vert_M)_{k_s}$ decomposes into
the direct sum of one-dimensional representations
with weights the $\gamma \cdot \chi \vert_M$, where 
$\gamma \in \Gamma$ (since the algebra
$k(\chi \vert_M) \otimes_k k_s$ decomposes  
into the direct sum of twists of $k_s$ by elements of
$\Gamma$). Using the analogous decomposition
for $\rho(\chi' \vert_M)_{k_s}$, it follows that
$\Gamma \cdot \chi' \vert_M = \Gamma \cdot \chi \vert_M$.
Hence $\chi' \in \Gamma \cdot \chi$. 
\end{proof}

\subsection{Blocks}
\label{subsec:blocks}

In this subsection, we study the decomposition 
of the category $H\tmod$ into blocks; the main result
(Proposition \ref{prop:reduc}) describes the structure
of each block in terms of twisted representations
of the largest unipotent quotient $U$. We start with
a key observation:

\begin{proposition}\label{prop:ext}
Let  $\chi,\chi' \in \rX(H)$. If
$\chi' \notin \Gamma \cdot \chi$ then
$\Ext^1_H(k(\chi),k(\chi')) = 0$.
\end{proposition}

\begin{proof}
By Lemma \ref{lem:def}, we have 
$\res_M^H \, k(\chi) \simeq n \, k(\chi \vert_M)$, 
$\res_M^H \, k(\chi') \simeq n' \, k(\chi' \vert_M)$
for positive integers $n,n'$; moreover, the
simple $M$-modules $k(\chi \vert_M)$,
$k(\chi' \vert_M)$ are non-isomorphic. 
Thus, every extension of $H$-modules
\[ 0 \longrightarrow k(\chi') \longrightarrow E
\longrightarrow k(\chi) \longrightarrow 0 \]
has a unique splitting as an extension of 
$M$-modules; in particular, the space 
$\Hom_M(k(\chi),E)$ is nonzero. As this space 
is a module under the unipotent group scheme 
$U = H/M$, it contains a nonzero $U$-invariant;
hence $\Hom_H(k(\chi),E) \neq 0$. 
As a consequence, the above extension splits.
\end{proof}

Combining Proposition \ref{prop:ext} with 
\cite[II.7.1]{Jantzen}, we obtain a decomposition
of the category of finite-dimensional $H$-modules 
into \emph{blocks},
\begin{equation}\label{eqn:blocks} 
H\tmod = \bigoplus_{\chi \in \rX(H)/\Gamma} H\tmod_{\chi}, 
\end{equation}
where $H\tmod_{\chi}$ denotes the full subcategory of $H\tmod$ 
consisting of the modules having all their composition
factors isomorphic to $k(\chi)$. Each $H\tmod_{\chi}$ 
is an abelian category; it has a unique simple object,
$k(\chi)$, which is its own commutant algebra
(Proposition \ref{prop:irr}). 
 
Taking for $\chi$ the trivial character, 
we obtain the \emph{principal block} $H\tmod_0$, 
consisting of the modules with trivial composition factors. 
By Lemma \ref{lem:def}, these are exactly the modules 
fixed pointwise by $M$, or equivalently the modules under
$H/M = U$. This yields an equivalence 
\begin{equation}\label{eqn:Umod} 
H\tmod_0 \simeq U\tmod, 
\end{equation} 
where the right-hand side is the category of 
\emph{unipotent representations}.

\begin{remark}\label{rem:blocks}
Lemma \ref{prop:ext} and the block decomposition 
(\ref{eqn:blocks}) easily imply that
$\Ext^i_H(k(\chi),k(\chi')) = 0$ for all $i \geq 1$ and 
$\chi,\chi' \in \rX(H)$ such that $\chi' \notin \Gamma \cdot \chi$.
Also, we have canonical isomorphisms
\[ \Ext^i_H(k(\chi),k(\chi)) \simeq
\rH^i(H, k(\chi)^{\vee} \otimes k(\chi)) \simeq
\rH^i(U, \End_M \, k(\chi)), \]
where the first equality holds by \cite[I.4.4]{Jantzen},
and the second one follows from the Hochschild-Serre
spectral sequence (see \cite[I.6.6]{Jantzen})
in view of the complete reducibility of representations of $M$.

If $k(\chi)$ is separable over $k$, then $k(\chi)$ 
is irreducible as an $M$-module, and hence 
$\End_M \, k(\chi) = k(\chi)$ by Proposition \ref{prop:irr}. 
This yields canonical isomorphisms
\[ \Ext^i_H(k(\chi),k(\chi)) \cong 
\rH^i(U,k) \otimes_k k(\chi) \quad (i \geq 0), \]
which also follow from Theorem \ref{thm:sep} below.
As $\rH^1(U,k) \neq 0$ for any nontrivial unipotent group 
$U$ (see e.g. \cite[II.3.7, IV.2.5]{DG}), it follows
that the category $H\tmod_{\chi}$ is semisimple
if and only if $U = 0$. 

For an arbitrary character $\chi$, it would be interesting 
to explicitly describe $\Ext^1_H(k(\chi),k(\chi))$, 
and to deduce an effective criterion for its vanishing
(which is equivalent to the category $H\tmod_{\chi}$
being semisimple). We will obtain a characterization
of the semisimplicity of $H\tmod_{\chi}$ by an alternative
approach, in Lemma \ref{lem:inj}.
\end{remark}

The block decomposition (\ref{eqn:blocks}) extends 
to a decomposition of the category of $H$-modules,
\[ 
H\tMod = \bigoplus_{\chi \in \rX(H)/\Gamma} H\tMod_{\chi}, 
\]
where $H\tMod_{\chi}$ consists of the direct limits of
objects of $H\tmod_{\chi}$. In particular, we obtain
a decomposition of the regular representation,
\[ \cO(H) = \bigoplus_{\chi \in \rX(H)/\Gamma} \cO(H)_{\chi}. \]

\begin{lemma}\label{lem:dec}

\begin{enumerate}

\item[{\rm (i)}]
For any $H$-module $V$, we have a natural isomorphism 
of $H$-modules
\[ \bigoplus_{\chi \in \rX(H)/\Gamma} 
\Hom_M(k(\chi \vert_M),V) 
\stackrel{\simeq}{\longrightarrow} V, 
\quad f \longmapsto f(1),\]
where $H$ acts on $\Hom_M(k(\chi \vert_M),V)$ 
via its action on $V$.

\item[{\rm (ii)}]
Each $\Hom_M(k(\chi \vert_M),V)$ is an object of
$H\tMod_{\chi}$. 

\item[{\rm (iii)}]
Each $\cO(H)_{\chi}$ is the injective hull of $k(\chi)$
in $H\tMod_{\chi}$.

\end{enumerate}

\end{lemma}

\begin{proof}
(i) As the $M$-module $\res_M^H \, V$ is semisimple
and the simple $M$-modules are exactly the 
$k(\chi \vert_M)$, where $\chi$ is uniquely determined 
up to the $\Gamma$-action, we have an isomorphism
\[ \bigoplus_{\chi \in \rX(H)/\Gamma} 
\Hom_M(k(\chi \vert_M),V) 
\otimes_{\End_M( k(\chi \vert_M))} k(\chi \vert_M) 
\stackrel{\simeq}{\longrightarrow} V, \]
which takes every $f \otimes v$ to $f(v)$.
Moreover, $\End_M( k(\chi \vert_M))$ is identified with 
$k(\chi \vert_M)$ acting by multiplication, in view of 
Proposition \ref{prop:irr}. This yields the assertion.

(ii) By construction, $\res_M^H \, \Hom_M(k(\chi \vert_M),V)$
is a direct sum of copies of $k(\chi \vert_M)$. This implies 
the statement in view of Lemma \ref{lem:def} (iii).

(iii) Since $\cO(H)$ is an injective object of $H\tMod$,
we see that $\cO(H)_{\chi}$ is an injective object
of $H\tMod_{\chi}$. So it suffices to show that 
$\cO(H)_{\chi}$ contains a unique copy of $k(\chi)$. But
\[ \Hom_H(k(\chi), \cO(H)_{\chi}) = 
\Hom_H(k(\chi), \cO(H)) \simeq k(\chi)^{\vee}, \]
where the isomorphism holds by Frobenius reciprocity
(see \cite[I.3.7]{Jantzen}). Thus, 
$\Hom_H(k(\chi), \cO(H)_{\chi})$ is a $k$-vector space
of dimension $[k(\chi) : k]$, and hence a vector space
of dimension $1$ under $\End_H\, k(\chi) = k(\chi)$. 
This yields the desired assertion.
\end{proof}

Next, choose $\chi \in \rX(H)$ and use the notation
$K,\eta,L$ of Lemma \ref{lem:def}. Recall from that lemma 
that $L$ is the separable closure of $k$ in $K$; 
in particular, the degree $[K:L]$ is a power of 
the characteristic exponent of $k$. 
We view $K$ as a simple $H$-module via $\rho(\chi)$; 
likewise, $L$ is a simple $M$-module via $\rho(\eta)$.

\begin{lemma}\label{lem:inj}

\begin{enumerate}

\item[{\rm (i)}]  
$\cO(H)_{\chi} \simeq \ind_M^H(L)$ as $H$-modules.

\item[{\rm (ii)}] The category $H\tmod_{\chi}$ is semisimple 
if and only if the group scheme $U$ is finite of order $[K:L]$.

\end{enumerate}

\end{lemma}

\begin{proof}
(i) We argue as in the proof of Lemma \ref{lem:dec} (iii). 
By Frobenius reciprocity (see \cite[I.3.4]{Jantzen}), we have 
\[ \Hom_H(V,\ind_M^H \, L)) \simeq \Hom_M(\res_M^H \, V,L) \]
for any $H$-module $V$. 
Since every $M$-module is semisimple, it follows that 
$\ind_M^H \, L$ is injective in $H\tMod$. Also, if $V$ is
simple and not isomorphic to $K$, then 
$\Hom_M(\res_M^H \, L,V) = 0$ by Lemma \ref{lem:def} (iii).
Thus, $\ind_M^H \, L$ is an object of $H\tMod_{\chi}$. 
Finally,
\[ \Hom_H(K,\ind_M^H \, L) \simeq \Hom_M(K,L) = \Hom_L(K,L), \]
where the equality follows from the fact that $\End_M(L) = L$
(Proposition \ref{prop:irr}). Thus, $\Hom_H(K,\ind_M^H(L))$
is a $L$-vector space of dimension $[K:L]$, and hence
a $k$-vector space of dimension $[K:k]$. Since the $k$-vector
space $\End_H(K)$ has dimension $[K:k]$ as well (Proposition
\ref{prop:irr} again), this yields the assertion in view of
\cite[I.3.18]{Jantzen}.

(ii) The semisimplicity of $H\tmod_{\chi}$ is equivalent to
that of $H\tMod_{\chi}$, and the latter holds if and only if
the injective cogenerator
$\cO(H)_{\chi}$ is semi-simple; by Lemma \ref{lem:dec} (iii), 
this is equivalent to the equality $\dim \cO(H)_{\chi} = [K : k]$. 
Also,
\[ 
\dim \cO(H)_{\chi} = \dim (\cO(H) \otimes_k L)^M 
= \dim_{\bar{k}} 
(\cO(H_{\bar{k}}) \otimes_{\bar{k}} L_{\bar{k}})^{M_{\bar{k}}}
= \dim_{\bar{k}} 
\cO(U_{\bar{k}}) \otimes_{\bar{k}} L_{\bar{k}},
\]
where the latter equality follows from the isomorphism
$H_{\bar{k}} \simeq U_{\bar{k}} \times M_{\bar{k}}$.
Thus, $\dim \cO(H)_{\chi} = [L:k] \, \dim \cO(U)$; 
this implies the assertion.
\end{proof}

\begin{example}\label{ex: blocks}
Assume that $k$ is imperfect and separably closed. We may then
choose $t \in k \setminus k^p$. Let $V \subset \bG_a \times \bG_a$ 
be the zero subscheme of $y^p - x - t x^p$. Then $V$ is a 
nontrivial $k$-form of $\bG_a$; in view of 
\cite[Lem.~9.4]{Totaro},
it follows that there exists a nontrivial extension
\[ \xi : \quad 0 \longrightarrow \bG_m \longrightarrow E
\longrightarrow V \longrightarrow 0. \]
On the other hand, the projection $x : V \to \bG_a$ lies 
in an exact sequence
\[ 0 \longrightarrow \alpha_p \longrightarrow V 
\longrightarrow \bG_a \longrightarrow 0, \]
where $\alpha_p$ denotes the kernel of the Frobenius
endomorphism of $\bG_a$. The pullback of $\xi$ by 
$\alpha_p \to V$ yields an extension
\[ 0 \longrightarrow \bG_m \longrightarrow H
\longrightarrow \alpha_p \longrightarrow 0, \]
which is nontrivial as well; indeed, the pullback map 
\[ \Ext^1_{\cC}(V,\bG_m) \longrightarrow 
\Ext^1_{\cC}(\alpha_p,\bG_m) \]
is injective in view of the vanishing of 
$\Ext^1_{\cC}(\bG_a,\bG_m)$ (which follows e.g. from
\cite[III.6.2.5, III. 6.5.1]{DG}). Thus, the canonical
character of $\bG_m$ extends uniquely to a character
$\chi$ of $H$, which is not defined over $k$. But 
$\chi$ is defined over $K := k(t^{1/p})$, since 
$V_K \simeq (\bG_a)_K$ and hence $\xi$ splits over $K$.
As $[K : k] = p$, it follows that $k(\chi) = K$.
Since $k(\chi \vert_{\bG_m}) = k$ and $\alpha_p$
has order $p$, Lemma \ref{lem:inj} (ii) yields that
the category $H\tmod_{\chi}$ is semisimple.
\end{example}

The exact sequence (\ref{eqn:H}) yields an exact
sequence of $L$-group schemes
\[ 0 \longrightarrow M_L \longrightarrow H_L 
\longrightarrow U_L \longrightarrow 0. \] 
Viewing $\eta$ as a morphism $M_L \to \bG_{m,L}$,
this yields in turn a pushout diagram of such group schemes
\begin{equation}\label{eqn:pushout} \xymatrix{
0 \ar[r] & M_L \ar[r] \ar[d]^{\eta} & H_L \ar[r] \ar[d] 
& U_L  \ar[r] \ar[d]^{\id} & 0 \\
0 \ar[r] & \bG_{m,L} \ar[r]  
& H_{\eta} \ar[r] & U_L \ar[r] & 0 \\
} \end{equation}
Also, every object $V$ of $H\tMod_{\chi}$ may be identified
with $\Hom_M(L,V)$ in view of Lemma \ref{lem:dec}.
Thus, $V$ has the structure of a $L$-vector space equipped
with a compatible action of $H$, i.e., of an $H_L$-module.
Since $M_L$ acts on $V$ via $\eta$, the representation of $H_L$ 
in $V$ factors uniquely through a representation of the pushout 
$H_{\eta}$. Denote by $F(V)$ the corresponding 
$H_{\eta}$-module, and by $H_{\eta}\tmod_1$ the block
of $H_{\eta}\tmod$ associated with the identity character 
of $\bG_{m,L}$. We may view this block as that of twisted 
$U_L$-modules, where the twist is defined by
the extension in the bottom line of (\ref{eqn:pushout}).

\begin{proposition}\label{prop:reduc}
The assignment $V \mapsto F(V)$ extends to an exact
functor, which yields an equivalence of categories 
\[ F : H\tmod_{\chi} \longrightarrow H_{\eta}\tmod_1. \]
\end{proposition}

\begin{proof}
Consider a morphism of $H$-modules $f: V \to W$.
This defines a map
\[ F(f) : \Hom_M(L,V) \longrightarrow \Hom_M(L,W),
\quad u \longmapsto f \circ u. \]
We claim that $F(f)$ is a morphism of $H_{\eta}$-modules.
It suffices to check that $F(f)$ is $L$-linear, where
$L$ acts on $\Hom_M(L,V)$, $\Hom_M(L,W)$ via multiplication 
on itself. But this follows by identifying the $L$-modules 
$V,W$ with direct sums of copies of $L$, and using the fact
that $\End_M(L) = L$ (Proposition \ref{prop:irr}). 

The claim defines the functor $F$ on morphisms; one may 
readily check that $F$ is exact, fully faithful and
essentially surjective.
\end{proof}

Next, we introduce a class of characters for which
the associated block has an especially simple structure.
We say that $\chi \in \rX(H)$ is \emph{separable}, 
if its field of definition is a separable extension of $k$. 
Also, recall that an $H$-module $V$ is said to be 
\emph{absolutely semisimple}, if the $H_{\bar{k}}$-module 
$V_{\bar{k}} := V \otimes_k {\bar{k}}$ is semisimple.

\begin{lemma}\label{lem:sep}
With the notation of Lemma \ref{lem:def}, 
the following conditions are equivalent for a character $\chi$:

\begin{enumerate}

\item[{\rm (i)}] $\chi$ is separable.

\item[{\rm (ii)}] $L = K$.

\item[{\rm (iii)}] The representation $\res_M^H \, \rho(\chi)$ is irreducible.

\item[{\rm (iv)}] The extension
$0 \to \bG_{m,L} \to H_{\eta} \to U_L \to 0$
splits.

\item[{\rm (v)}] The $H$-module $K$ is absolutely semisimple.

\end{enumerate}

\end{lemma}

\begin{proof}
The equivalences 
(i)$\Leftrightarrow$(ii)$\Leftrightarrow$(iii) are obvious.
Also, in view of the commutative diagram (\ref{eqn:pushout}),
the splittings of the extension in (iv) may be
identified with the morphisms $H_L \to \bG_{m,L}$
extending the identity character of $\bG_{m,L}$
i.e., with the morphisms $H \to \rR_{L/k}(\bG_{m,L})$
extending $\rho(\eta): M \to \rR_{L/k}(\bG_{m,L})$; 
this implies the equivalence (ii)$\Leftrightarrow$(iv).

We now show the equivalence (i)$\Leftrightarrow$(v).
If $\chi$ is separable, then the $\bar{k}$-algebra
$K \otimes_k \bar{k}$ is isomorphic to a product
of copies of $\bar{k}$. This yields a decomposition of
the corresponding $H_{\bar{k}}$-module into
$H_{\bar{k}}$-stable lines; in particular, this module
is semisimple. 

Conversely, if the $H$-module $K$ is absolutely semisimple,
then so is its commutant algebra. By Proposition
\ref{prop:irr}, it follows that the field extension $K/k$
is separable. 
\end{proof}

\begin{theorem}\label{thm:sep}
Let $\chi$ be a separable character of $H$, with field
of definition $K$. Then the abelian category 
$H\tmod_{\chi}$ is equivalent to $U_K\tmod$.
\end{theorem}

\begin{proof}
We may replace $k$ with $K = L$ and $H$ with 
$H_{\eta}$ in view of Proposition \ref{prop:reduc}. 
Then Lemma \ref{lem:sep} yields an isomorphism
$H \simeq \bG_m \times U$. Denote by $k_1$ 
the $1$-dimensional $H$-module with weight $1$;
then the assignment $V \mapsto V \otimes k_1$
extends to the desired equivalence 
\[ U\tmod = H\tmod_0 \stackrel{\simeq}{\longrightarrow}
H\tmod_1. \]
\end{proof}

\section{Isogenies}
\label{sec:isos}

\subsection{Functorial properties of universal affine covers}
\label{subsec:fun}

Consider a morphism of abelian varieties $\varphi: A \to B$.

\begin{proposition}\label{prop:fun}

\begin{enumerate}

\item[{\rm (i)}] There are unique morphisms
$G(\varphi) : G_A \to G_B$, $H(\varphi) : H_A \to H_B$  
such that the diagram of extensions
\begin{equation}\label{eqn:isogeny} \xymatrix{
0 \ar[r] & H_A \ar[r] \ar[d]^{H(\varphi)} 
& G_A \ar[r]^-{f_A} \ar[d]^{G(\varphi)} 
& A  \ar[r] \ar[d]^{\varphi} & 0 \\
0 \ar[r] & H_B \ar[r]  & G_B  \ar[r]^-{f_B} & B \ar[r] & 0 \\
} \end{equation}
commutes. If $\varphi =n_A$ for some integer $n$, then 
$G(\varphi) = n_{G_A}$, $H(\varphi) = n_{H_A}$.
Also, the formations of $G(\varphi)$, $H(\varphi)$
commute with base change under field extensions.

\item[{\rm (ii)}] For any morphism
of abelian varieties $\psi : B \to C$, we have
$G(\psi \circ \varphi) = G(\psi) \circ G(\varphi)$, 
$H(\psi \circ \varphi) = H(\psi) \circ H(\varphi)$.

\item[{\rm (iii)}] If $\varphi$ is an isogeny,
then $G(\varphi)$ is an isomorphism. Moreover,
there is an exact sequence of commutative affine 
group schemes
\begin{equation}\label{eqn:H_AH_B} 
0 \longrightarrow H_A 
\stackrel{H(\varphi)}{\longrightarrow} H_B
\longrightarrow N \longrightarrow 0, 
\end{equation}
where $N := \Ker(\varphi)$.

\end{enumerate}

\end{proposition}

\begin{proof}
(i) The existence of the morphisms $G(\varphi)$, $H(\varphi)$ 
follows from the fact that $G_A$ is projective in the category
$\tcC$ of commutative quasi-compact group schemes. For the
uniqueness, just note that every morphism $G_A \to H_B$
is zero, since $\cO(G_A) = k$ and $H_B$ is affine. 
This uniqueness property implies the assertion on 
multiplication maps. Finally, the assertion on field 
extensions follows from Proposition \ref{prop:field}. 

(ii) This follows again from the uniqueness in (i).

(iii) There exists an isogeny $\psi : B \to A$ such that
$\psi \circ \varphi$ is the multiplication map $n_A$
for some positive integer $n$, and $\varphi \circ \psi = n_B$. 
Then $G(\psi) \circ G(\varphi) = n_{G_A}$ by (ii); moreover,
$n_{G_A}$ is an isomorphism in view of \cite[Lem.~3.4]{Br18}.
Likewise, $G(\varphi) \circ G(\psi) = n_{G_B}$ is an
isomorphism. This yields the assertion on $G(\varphi)$,
and in turn that on $H(\varphi)$ by applying the snake
lemma to the commutative diagram (\ref{eqn:isogeny}).
\end{proof}

Next, assume that $\varphi : A \to B$ is an isogeny 
and let $N:= \Ker(\varphi)$. Then the exact sequence  
(\ref{eqn:H_AH_B}) defines induction and restriction functors
(see \cite[I.3]{Jantzen} for details),
\[ 
\ind = \ind_{H_A}^{H_B} : H_A\tmod \to H_B\tmod,
\quad \res = \res_{H_A}^{H_B} : H_B\tmod \to H_A\tmod. \]

\begin{theorem}\label{thm:isogeny}

\begin{enumerate}
 
\item[{\rm (i)}] For any homogeneous vector bundle
$F$ on $B$, the pullback $\varphi^*(F)$ is a homogeneous
vector bundle on $A$. Moreover, the assignment 
$F \mapsto \varphi^*(F)$ yields an exact functor
$\varphi^* : \HVec_B \to \HVec_A$ which fits in a commutative
square
\[ \xymatrix{
H_B\tmod \ar[r]^{\res} \ar[d]^{\cL_B} & H_A\tmod \ar[d]^{\cL_A} \\
\HVec_B \ar[r]^{\varphi^*} & \HVec_A. \\
} \]

\item[{\rm (ii)}] For any homogeneous vector bundle
$E$ on $A$, the push-forward $\varphi_*(E)$ is a homogeneous
vector bundle on $B$. Moreover, the assignment 
$E \mapsto \varphi_*(E)$ yields an exact functor
$\varphi_* : \HVec_A \to \HVec_B$ which fits in a commutative
square
\[ \xymatrix{
H_A\tmod \ar[r]^{\ind} \ar[d]^{\cL_A} & H_B\tmod \ar[d]^{\cL_B} \\
\HVec_A \ar[r]^{\varphi_*} & \HVec_B. \\
} \]

\end{enumerate}

\end{theorem}

\begin{proof}
(i) Let $a \in A(\bar{k})$. Since $\varphi$
is a group homomorphism, we have
$\tau_a^* \varphi^*(F) \simeq 
\varphi^* \tau_{\varphi(a)}^*(F)$.
As $F$ is homogeneous, it follows that
$\tau_a^* \varphi^*(F) \simeq \varphi^*(F)$.
So $\varphi^*(F)$ is homogeneous as well.

Clearly, the assignment $F \mapsto \varphi^*(F)$
extends to an exact functor $\Vec_B \to \Vec_A$,
and hence to an exact functor 
$\varphi^* : \HVec_B \to \HVec_A$. The commutativity of 
the displayed square follows readily from the 
definitions.

(ii) Since $\varphi$ is finite and flat, it yields
an exact functor $\varphi_* : \Vec_A \to \Vec_B$. 
Let $b \in B(\bar{k})$, and choose $a \in A(\bar{k})$ 
such that $b = \varphi(a)$; then the diagram
\[ \xymatrix{
A \ar[r]^{\tau_a} \ar[d]^{\varphi} & A \ar[d]^{\varphi}  \\
B \ar[r]^{\tau_b}  & B \\
} \]
commutes, where the horizontal arrows are 
isomorphisms. It follows that 
$\tau_b^* \varphi_*(E) \simeq \varphi_* \tau_a^*(E)$.
Thus, $\varphi_*(E)$ is homogeneous. 

By the projection formula, $\varphi_*$ is right
adjoint to $\varphi^*$. Also, recall that $\ind$
is right adjoint to $\res$ (see \cite[I.3.4]{Jantzen}).
So the desired isomorphism of functors
$\cL_B \circ \ind \simeq \varphi_* \circ \cL_A$
follows from the isomorphism 
$\cL_A \circ \res \simeq \varphi^* \circ \cL_B$
in view of the uniqueness of adjoints.
\end{proof}

Next, recall the dual exact sequence
\begin{equation}\label{eqn:dual} 
0 \longrightarrow \rD(N) \longrightarrow \wB
\stackrel{\wvarphi}{\longrightarrow} \wA
\longrightarrow 0, 
\end{equation}
where $\wvarphi$ denotes the dual isogeny, and
$\rD(N)$ the Cartier dual of the finite group scheme 
$N$. Also, recall the natural isomorphism
$\Hom_{\tcL}(H_{A_{k'}},\bG_{m,k'}) \simeq \wA(k')$
for any field extension $k'/k$ (Lemma \ref{lem:H_A}). 
Thus, we may identify the character group 
$\rX(H_A)$ with $\wA(\bar{k})$, and the field of
definition of any $x \in \rX(H_A)$ with the residue
field $k(x)$. In particular, the separable characters
(as defined in \S \ref{subsec:blocks}) correspond to 
the points of $\wA(k_s)$. Also, the morphism 
$H(\varphi) : H_A \to H_B$ defines a pullback map 
$H(\varphi)^* : \rX(H_B) \to \rX(H_A)$.

\begin{lemma}\label{lem:isos}
For any $y \in \wB(\bar{k})$, we have
$H(\varphi)^*(y) = \wvarphi(y)$ in $\wA(\bar{k})$.
\end{lemma}

\begin{proof}
Consider first the case where $y \in \wB(k)$.
We then have a commutative diagram with exact rows
\[ \xymatrix{
0 \ar[r] & H_A \ar[r] \ar[d]^{H(\varphi)} 
& G_A \ar[r]^-{f_A} \ar[d]^{G(\varphi)} 
& A  \ar[r] \ar[d]^{\varphi} & 0 \\
0 \ar[r] & H_B \ar[r] \ar[d]^{y}  & G_B  
\ar[r]^-{f_B} \ar[d] & B \ar[r] \ar[d]^{\id} & 0 \\
0 \ar[r] & \bG_m \ar[r] & G \ar[r] & B \ar[r] & 0, \\
} \]
where the bottom line is obtained by pushout. This 
yields a commutative diagram 
\[ \xymatrix{
\Hom_{\tcL}(H_B,\bG_m) 
\ar[r]^-{\partial(B)} \ar[d]^{H(\varphi)^*} 
& \Ext^1_{\cC}(B,\bG_m) \ar[d]^{\varphi^*}  \\
\Hom_{\tcL}(H_A,\bG_m) \ar[r]^-{\partial(A)}  
&  \Ext^1_{\cC}(A,\bG_m), \\
} \]
where the horizontal arrows are isomorphisms
given by pushout. Also, the right vertical arrow 
may be identified with $\wvarphi : \wB(k) \to \wA(k)$. 
So we may rewrite the above square as
\[ \xymatrix{
\Hom_{\tcL}(H_B,\bG_m) 
\ar[r]^-{\simeq} \ar[d]^{H(\varphi)^*} 
& \wB(k) \ar[d]^{\wvarphi}  \\
\Hom_{\tcL}(H_A,\bG_m) \ar[r]^-{\simeq} &  \wA(k), \\
} \]
which yields the assertion in that case.

In the general case, we argue similarly by replacing $k$ 
with the residue field $k(y)$, and using compatibility 
with field extensions (Proposition \ref{prop:fun}). 
\end{proof}

Still considering $x \in \rX(H_A) = \wA(\bar{k})$ with 
residue field $k(x)$, the representation 
of $H_A$ in $k(x)$ (constructed in \S \ref{subsec:irr}) 
yields an associated homogeneous vector bundle 
$E(x) := \cL_A(k(x))$ over $A$, of rank $[ k(x) : k ]$. 
By Theorem \ref{thm:equiv} and Proposition \ref{prop:irr}, 
$E(x)$ is irreducible; moreover, every irreducible 
homogeneous vector bundle on $A$ is obtained 
in this way. Also, given $x' \in \wA(\bar{k})$, we have 
$E(x') \simeq E(x)$ if and only if 
$x' \in \Gamma \cdot x$ (Proposition \ref{prop:gal}).
Otherwise, $\Ext^i_A(E(x),E(x')) = 0$ for any $i \geq 0$,
by Theorem \ref{thm:ext} and Proposition \ref{prop:ext}.

If $x \in \wA(k)$, then $E(x)$ is just the corresponding
(algebraically trivial) line bundle on $A$. For an 
arbitrary residue field $k(x) =: K$, we obtain a line
bundle $L$ on $A_K$, and hence a vector bundle
$\rR_{K/k}(L)$ on $\rR_{K/k}(A_K)$, of rank 
$[K : k]$; one may check that $E(x)$ is the pullback 
of that vector bundle under the canonical immersion
$j_A : A \to \rR_{K/k}(A_K)$ (see \cite[A.5.7]{CGP}).

Likewise, we have irreducible homogeneous vector bundles
$F(y)$ on $B$, indexed by the $\Gamma$-orbits in
$\wB(\bar{k})$. We now determine their pullback to $A$:

\begin{proposition}\label{prop:pull}
Let $y \in \wB(\bar{k})$ and set $x := \wvarphi(y)$.
Then we have $k(x) \subset k(y)$ and
$\varphi^* F(y) \simeq [k(y) : k(x)] \, E(x)$.
\end{proposition}

\begin{proof}
Clearly, $k(x) \subset k(y)$. Also, by Theorem 
\ref{thm:isogeny}, $\varphi^* F(y)$ is the homogeneous 
vector bundle associated with the representation of $H_A$ 
in $k(y)$ via the composition 
\[ H_A \stackrel{H(\varphi)}{\longrightarrow}
H_B \stackrel{\rho(y)}{\longrightarrow}
\rR_{k(y)/k}(\bG_{m,k(y)}). \]
The associated character 
$(H_A)_{k(y)} \to \bG_{m,k(y)}$ equals $x$ in view 
of Lemma \ref{lem:isos}. Thus, $H_A$ acts on $L$ via
the corresponding morphism $H_A \to \rR_{k(x)/k}(\bG_{m,k(x)})$.
It follows that 
$\res^{H_B}_{H_A} \, \rho(y) \simeq [k(y) : k(x)] \, \rho(x)$;
this translates into the desired isomorphism.
\end{proof}

The block decomposition (\ref{eqn:blocks}) of $H_A$-mod 
yields a decomposition
\[ 
\HVec_A = \bigoplus_{x \in \wA(\bar{k})/\Gamma} \HVec_{A,x}. 
\]
As a direct consequence of Proposition \ref{prop:pull}, 
we obtain:

\begin{corollary}\label{cor:pull}
$\varphi^* : \HVec_B \to \HVec_A$ takes $\HVec_{B,y}$ to 
$\HVec_{A,\wvarphi(y)}$ for any $y \in \wB(\bar{k})$.
Moreover, $\varphi^*$  preserves semisimplicity.
\end{corollary}

Next, we describe the pushforward of irreducible 
homogeneous bundles:

\begin{proposition}\label{prop:push}

\begin{enumerate}

\item[{\rm (i)}] Let $x \in \wA(\bar{k})$ and consider 
the block decomposition
\[ \varphi_* \, E(x) = 
\bigoplus_{y \in \wB(\bar{k})/\Gamma} \, F_{x,y}. \]
Then $F_{x,y} \neq 0$ if and only if 
$\wvarphi(y) \in \Gamma \cdot x$. Under this assumption,
$F_{x,y}$ contains a unique copy of $F(y)$; in particular,
$F_{x,y}$ is indecomposable.

\item[{\rm (ii)}] $\varphi_*$ preserves semisimplicity
if and only if $\wvarphi$ is separable.
Under this assumption, $F_{x,y}$ is irreducible for all
$x,y$.

\end{enumerate}

\end{proposition}

\begin{proof}
(i) Let $y \in \wB(\bar{k})$. Then
\[ \Hom_{\HVec_B}(F(y),\varphi_* E(x)) \simeq
\Hom_{\HVec_A}(\varphi^* F(y), E(x)) \]
\[ \simeq [ k(y) : k(\wvarphi(y) ] \, 
\Hom_{\HVec_A}(E(\wvarphi(y)),E(x)), \] 
where the first isomorphism holds by adjunction,
and the second one follows from Proposition
\ref{prop:pull}. Thus, 
\[ \dim_k \Hom_{\HVec_B}(F(y),\varphi_* E(x)) = 
[k(y) : k( \gamma \cdot x) ] \, [k(\gamma \cdot x) : k]
= [k(y) : k] \]
if $\wvarphi(y) = \gamma \cdot x$ for some
$\gamma \in \Gamma$, and this dimension is zero otherwise. 
Since the $k$-vector space 
$\Hom_{\HVec_B}(F(y),\varphi_* E(x))$ is a module under
$\End_{\HVec_B} k(y) \simeq k(y)$, we see that
$F_{x,y}$ contains a unique copy of $F(y)$ if 
$\wvarphi(y) \in \Gamma \cdot x$, and is zero otherwise.

(ii) Assume that $\varphi_*$ preserves semisimplicity; 
in particular, the homogeneous bundle $\varphi_*(\cO_A)$ 
is semisimple. By Theorems \ref{thm:equiv} and 
\ref{thm:isogeny}, it follows that the $H_B$-module 
$\cO(H_B/H_A) = \cO(N)$ is semisimple as well. 
Equivalently, $N$ is of multiplicative type (see 
\cite[II.2.2.2, IV.3.3.6]{DG}), i.e., the Cartier dual
$\rD(N)$ is \'etale. In view of the exact sequence
(\ref{eqn:dual}), this means that $\wvarphi$ is separable.

Conversely, assume that $\wvarphi$ is separable; then
$N$ is of multiplicative type, as seen by reverting the
above arguments. Note that $\varphi_*$ preserves 
semisimplicity if and only if so does $\ind_{H_A}^{H_B}$ 
(Theorem \ref{thm:isogeny}). We also need a general 
observation: let $H$ be an affine group scheme
and $V$ an $H$-module; then $V$ is semisimple
if and only if the $H_{k_s}$-module $V_{k_s}$ is 
semisimple. Indeed, $V$ is a module under the 
finite-dimensional algebra $C^{\vee}$ constructed 
in the proof of Proposition \ref{prop:irr}; 
moreover, $V$ is semisimple as an $H$-module 
if and only if it is semisimple as a $C^{\vee}$-module 
(as follows from \cite[II.2.2.4]{DG}).
Since the formation of $C^{\vee}$ commutes with field 
extensions, this implies the observation by using the 
invariance of semisimplicity under separable extensions 
(see \cite[VIII.13.4]{Bourbaki}). 

In view of this observation, we may assume that
$k$ is separably closed. The fiber of $\wvarphi$
at any $x \in \wA(\bar{k})$ has $n$ distinct $\bar{k}$-points
$y_1,\ldots,y_n$, where 
$n = \deg(\wvarphi) = \deg(\varphi)$. Moreover,
$k(y_i) = k(x)$ for $i = 1,\ldots,n$: indeed,
$k(y_i)$ is a separable extension of $k(x)$
as $\wvarphi$ is \'etale, and $k(y_i)/k(x)$
is purely inseparable as $k = k_s$. By (i),
$\varphi_* \, E(x)$ contains a sub-bundle
isomorphic to $\oplus_{i = 1}^n \, F(y_i)$.
Since 
\[ \rk \, \varphi_* \, E(x) = n \, [ k(x) : k ]
= \sum_{i=1}^n [ k(y_i) : k], \]
it follows that 
$\varphi_* \, E(x) \simeq \oplus_{i = 1}^n \, F(y_i)$.
This completes the proof of preservation of
semisimplicity under $\varphi_*$. The irreducibility 
of $F_{x,y}$ follows in view of (i).
\end{proof}

\begin{remark}\label{rem:push}
To determine the pushforward of irreducible homogeneous
vector bundles under an arbitrary isogeny $\varphi$,
consider the exact sequence
\[ 0 \longrightarrow M \longrightarrow N 
\longrightarrow U \longrightarrow 0, \]
where $M$ is of multiplicative type and $U$ is unipotent.
This yields a factorization 
$\varphi = \varphi_u \circ \varphi_m$, 
where $\varphi_m: A \to A/M$ has kernel $M$ and 
$\varphi_u:  A/M \to B$ has kernel $U$, and a dual 
factorization $\wvarphi = \wvarphi_m \circ \wvarphi_u$,
where $\wvarphi_m$ is separable and $\wvarphi_u$
is purely inseparable. As the pushforward 
$(\varphi_m)_*\, E(x)$ is described by Proposition 
\ref{prop:push}, we may replace $\varphi$ with 
$\varphi_u$, and hence assume that $N = U$
is unipotent. Then $\wvarphi$ is bijective on 
$\bar{k}$-points; by Proposition \ref{prop:push}
again, it follows that $\varphi_* \, E(x)$
is indecomposable for any $x \in \wA(\bar{k})$.
\end{remark}

\subsection{Unipotent vector bundles}
\label{subsec:unip}

Recall that a vector bundle $E$ over $A$ is \emph{unipotent} if it
admits a filtration 
$0 = E_0 \subset E_1 \subset \cdots \subset E_n = E$,
where each $E_i$ is a sub-bundle and $E_i/E_{i-1} \simeq \cO_A$
for $i = 1, \ldots, n$. The unipotent vector bundles form a full
subcategory $\UVec_A$ of $\Vec_A$.

\begin{theorem}\label{thm:unip}

\begin{enumerate}

\item[{\rm (i)}] $\UVec_A = \HVec_{A,0} \simeq U_A\tmod$;
in particular, $\UVec_A$ has a unique 
(up to isomorphism) simple object, $\cO_A$. Moreover,
we have isomorphisms of graded algebras
\[ \Ext^*_{\UVec_A}(\cO_A,\cO_A) 
\simeq \rH^*(A,\cO_A) \simeq 
\Lambda^* \rH^1(A,\cO_A). \]

\item[{\rm (ii)}] $\UVec_A$ is an abelian tensor subcategory
of $\Vec_A$, stable under extensions and direct summands.

\item[{\rm (iii)}] We have an equivalence of abelian categories 
$\HVec_{A,x} \simeq \UVec_{A_{k(x)}}$ for any $x \in \wA(k_s)$.

\end{enumerate}

\end{theorem}

\begin{proof}
(i) By Corollary \ref{cor:ext}, every unipotent vector bundle
is homogeneous. This yieds the equality 
$\UVec_A = \HVec_{A,0}$; the latter category is equivalent 
to $U_A\tmod$ by Theorem \ref{thm:equiv} and
(\ref{eqn:Umod}). This shows the first assertion. The second 
assertion is obtained by combining Theorem \ref{thm:ext} 
and \cite[Thm.~VII.10]{Se97}.

(ii) This follows from (i) by using Corollary \ref{cor:abe}. 

(iii) This is a direct consequence of Theorem \ref{thm:sep}. 
\end{proof}

\begin{remark}\label{nilpotent}
The unipotent vector bundles over an elliptic curve $A$
have been determined by Atiyah (when $k$ is algebraically
closed), in the process of his description of all vector bundles
over $A$; see \cite{Atiyah}. In particular, there is a unique
indecomposable unipotent bundle of rank $r$ for any integer 
$r \geq 0$. The decomposition of the tensor product 
of any two such bundles has been determined in \cite{Atiyah} 
when $p = 0$; the case where $p >0$ has been treated
much more recently by Schroer (see \cite{Schroer}).

Returning to an abelian variety $A$ over an arbitrary field 
$k$, recall that $U_A \simeq (\bG_a)^g$ if $p = 0$. 
As a consequence, the category $\UVec_A$ is equivalent 
to the category with objects the tuples $(r,X_1,\ldots,X_g)$, 
where $r$ is a non-negative integer and $X_1,\ldots,X_g$
are commuting nilpotent $r \times r$ matrices with
coefficients in $k$; the morphisms from 
$(r,X_1,\ldots,X_g)$ to $(s,Y_1,\ldots, Y_g)$ are 
the $s\times r$ matrices $Z$ with coefficients in $k$ 
such that $Z X_i = Y_i Z$ for $i = 1,\ldots,g$. 
An explicit description of the isomorphism classes
of such tuples is well-known for $g = 1$, via the Jordan
canonical form (which gives back the above results of
Atiyah). But the higher-dimensional case is quite open; 
see \cite{HH} for a study of the moduli space of ``regular'' 
tuples.

This description of $\UVec_A$ in terms of linear algebra
extends to an ordinary abelian variety $A$ over 
a separably closed field $k$ of characteristic $p > 0$, 
since we then have $U_A \simeq (\bZ_p)^g_k$ 
by (\ref{eqn:U_A}).

\end{remark}

\begin{proposition}\label{prop:uisogeny}
Let $\varphi : A  \to B$ be an isogeny with kernel $N$.

\begin{enumerate}

\item[{\rm (i)}] $\varphi^* : \HVec_B \to \HVec_A$ takes 
$\UVec_B$ to $\UVec_A$. If $\varphi$ is separable, then 
this yields an equivalence of categories 
$\UVec_B \simeq \UVec_A$.

\item[{\rm (ii)}] $\varphi_* : \HVec_A \to \HVec_B$ takes $\UVec_A$ 
to $\bigoplus_{y \in \rX(N)/\Gamma} \HVec_{B,y}$.
In the resulting decomposition of $\varphi_*(\cO_A)$,
each summand is indecomposable.

\end{enumerate}

\end{proposition}

\begin{proof}
(i) The first assertion holds since $\varphi^*$ is
exact and takes $\cO_B$ to $\cO_A$. The second 
assertion follows from Theorem \ref{thm:isogeny} 
in view of the exact sequence (\ref{eqn:H_AH_B}), 
which yields an isomorphism 
$U_A \stackrel{\simeq}{\longrightarrow} U_B$.

(ii) This follows similarly from the exactness
of $\varphi_*$ and Proposition \ref{prop:push}.
\end{proof}

\begin{remark}\label{rem:ST}
Assume that $p > 0$; then the $n$th relative Frobenius morphism 
$\rF_A^n$ is a purely inseparable isogeny of degree $p^{ng}$. 
By Theorem \ref{thm:isogeny} and Proposition \ref{prop:uisogeny}, 
the decomposition of $(\rF_A^n)_* \, \cO_A$ into indecomposable 
summands corresponds to the block decomposition of $\cO(N)$, 
where $N := \Ker(\rF_A^n)$ is an infinitesimal group scheme
of order $p^{ng}$. Denoting by $r$ the $p$-rank of $A$, 
the largest subgroup scheme of multiplicative type of $N$ 
is a $k$-form of $(\mu_{p^n})^r$ in view of  (\ref{eqn:ptors}). 
When (say) $k$ is separably closed, it follows that each block
of $\cO(N)$ has dimension $p^{n(g-r)}$; equivalently,
$(\rF_A^n)_* \, \cO_A$ is the direct sum of $p^{nr}$
indecomposable summands of dimension $p^{n(g-r)}$. 
This gives back a recent result of Sannai and Tanaka 
(see \cite[Thm.~1.2]{ST}).

The analogous decomposition of $(\rF_A^n)_*(L)$, 
obtained in \cite[Thm.~5.3]{ST} for any $L \in \wA(k)$,  
can also be derived from Proposition \ref{prop:push}
and Remark \ref{rem:push}.
\end{remark}

Still assuming $p >0$, and denoting by
$\rV^n_A : A^{(p^n)} \to A$ the $n$th Verschiebung, 
we obtain a characterization of unipotent vector bundles 
which refines a result of Miyanishi (see 
\cite[Rem.~2.4]{Miyanishi}):

\begin{proposition}\label{prop:V}
Let $E$ be a vector bundle over $A$. Then
$E$ is unipotent if and only if $(\rV^n_A)^*\, E$ 
is trivial for $n \gg 0$.
\end{proposition}

\begin{proof}
Assume that $E$ is unipotent. Then there exists a
finite-dimensional representation
$\rho: H_A \to \GL(V)$ such that $E = \cL_A(V)$ and
$\Ker(\rho) \supset M_A$. Thus, $H_A/\Ker(\rho)$
is a commutative unipotent algebraic group, and
hence there exists $n_0$ such that
$\rV^n_{H_A/\Ker(\rho)} = 0$ for $n \geq n_0$ 
(see \cite[IV.3.4.11]{DG}). As a consequence, 
$\rho \circ \rV^n_{H_A}$ is trivial for $n \geq n_0$. 
In view of the commutative diagram with exact rows
\[ \xymatrix{
0 \ar[r] & H_A^{(p^n)} \ar[r] \ar[d]^{\rV^n_{H_A}} 
& G_A^{(p^n)} \ar[r]^-{f_A^{(p^n)}} \ar[d]^{\rV^n_{G_A}} 
& A^{(p^n)}  \ar[r] \ar[d]^{\rV^n_A} & 0 \\
0 \ar[r] & H_A \ar[r]  & G_A  \ar[r]^-{f_A} & A \ar[r] & 0, \\
} \]
it follows that $(\rV^n_A)^* \, E$ is trivial 
for $n \geq n_0$.

Conversely, assume that $(\rV^n_A)^* \, E$ is trivial
for $n \geq n_0$.
By Theorem \ref{thm:isogeny} (i), it follows that
$E \simeq G_A \times^{H_A} V$, where $V$ is a 
finite-dimensional $H_A$-module which restricts
trivially to $H_A^{(p^n)}$ via $H(\rV^n_A)$. Thus, 
\[ E \simeq G_A/H_A^{(p^n)} 
\times^{H_A/H_A^{(p^n)}} V 
\simeq A^{(p^n)} \times^{\Ker(\rV^n_A)} V. \]
Since $\Ker(\rV^n_A)$ is unipotent, it follows
that $E$ is unipotent as well.
\end{proof}

\begin{remark}\label{rem:ess}
The \emph{essentially finite} vector bundles $E$
over $A$, i.e., those such that $f^*(E)$ is trivial
for some torsor $f : X \to A$ under a finite group scheme,
admit a similar characterization. Indeed, as shown by Nori
in \cite{No83}, for any such bundle $E$, there exists $n > 1$ 
such that $n_a^*(E)$ is trivial. 
As a consequence, the essentially finite vector bundles
are exactly the iterated extensions of irreducible homogeneous
vector bundles associated with torsion points of
$\wA(\bar{k})$.

Also, $U_A$ is canonically isomorphic to the 
``nilpotent fundamental group scheme'' $U(A,0)$ introduced
and studied by Nori in \cite[Chap.~IV]{No82} in the more general
setting of pointed schemes of finite type. This follows 
from [loc.~cit., Prop.~1]; note that the ``nilpotent group 
schemes'' considered there are exactly the unipotent group 
schemes in the sense of \cite[IV.2.2]{DG}.   
\end{remark}

\medskip

\noindent
{\bf Acknowledgments.} Many thanks to Corrado De Concini, 
Pedro Luis del \'Angel Rodr\'{\i}guez, Ga\"el R\'emond, 
\'Alvaro Rittatore, Catharina Stroppel, Tam\'as Szamuely
and Angelo Vistoli for very helpful discussions. Also, thanks 
to the anonymous referee for his/her careful reading
and valuable comments.

\bibliographystyle{amsalpha}

\end{document}